\documentclass[11pt]{article}
\usepackage{amsmath}
\usepackage{amsthm}
\usepackage{amssymb}
\usepackage{amscd}
\usepackage[bookmarks=true]{hyperref}
\newcommand{\C}{\mathbb{C}}
\newcommand{\R}{\mathbb{R}}
\newcommand{\Z}{\mathbb{Z}}

\theoremstyle{plain}
\newtheorem{theorem}{Theorem}[section]
\newtheorem{proposition}[theorem]{Proposition}
\newtheorem{lemma}[theorem]{Lemma}
\newtheorem{Corollary}[theorem]{Corollary}

\theoremstyle{definition}

\newtheorem{remark}[theorem]{Remark}

\newtheorem{example}[theorem]{Example}

\makeatletter
\@addtoreset{equation}{section}
\makeatother
\newcounter{constantLABEL}
\newcommand{\cref}[1]{C_{\ref{#1}}}

\newcounter{constantslabel}
\begin{document}


\title{A construction of $Spin(7)$-instantons} 


\author{Y. Tanaka}

\date{}


\maketitle


\begin{abstract}

Joyce constructed examples of compact eight-manifolds with holonomy 
$Spin(7)$, starting with a 
Calabi--Yau four-orbifold with isolated singular points of a special kind. 
That construction can be seen as the gluing of ALE 
$Spin(7)$-manifolds to each singular point of the Calabi--Yau 
four-orbifold divided by an anti-holomorphic 
involution fixing only the singular points.  

On the other hand,  
there are higher-dimensional analogues of 
anti-self-dual instantons in four dimensions on $Spin(7)$-manifolds, 
which are called $Spin(7)$-instantons. 
They are minimizers of the Yang--Mills action, and 
the $Spin(7)$-instanton equation together with a gauge fixing condition 
forms an elliptic system. 

In this article, we construct $Spin(7)$-instantons on the examples of 
compact $Spin(7)$-manifolds above, starting with 
Hermitian--Einstein connections on the Calabi--Yau four-orbifolds 
and ALE spaces. 
Under some assumptions on the Hermitian--Einstein connections, we glue 
them together to obtain $Spin(7)$-instantons 
on the compact $Spin(7)$-manifolds. 
We also give a simple example of our construction.

\end{abstract}


\markboth{}
{A construction of $Spin(7)$-instantons}

\section{Introduction}

This article is about a construction of $Spin(7)$-instantons on examples 
of compact Riemannian 8-manifolds with holonomy $Spin(7)$. 
We construct those on Joyce's $Spin(7)$-manifolds of the second type, 
namely on a resolution 
of the quotient of a Calabi--Yau four-orbifold by an 
anti-holomorphic involution fixing only the singular points.

A $Spin(7)$-manifold is an 8-dimensional Riemannian manifold with
holonomy 
contained in the group $Spin(7)$. 
The holonomy group $Spin(7)$ is one of exceptional cases 
(the other is the group $G_2$) of Berger's 
classification of Riemannian holonomy groups of simply-connected, 
irreducible, non-symmetric Riemannian manifolds  
\cite{MR0079806}. 
Later metrics with holonomy $Spin(7)$ (and $G_2$ as well) were obtained by 
Bryant \cite{MR916718}, Bryant--Salamon \cite{MR1016448} 
for non-compact cases, and by Joyce  \cite{MR1383960}, 
\cite{MR1776092}, \cite{MR1787733}    
for compact cases.

There are two types in Joyce's constructions of compact 
$Spin(7)$-manifolds, namely, the construction of the metrics on  
\begin{enumerate}
 \item[(I)] the resolution of $T^8 /\Gamma$, where $T^8$ is a torus 
 and $\Gamma$ is a finite subgroup of automorphisms of $T^8$ 
\cite{MR1383960}, (\cite{MR1787733}). 
 \item[(I$\!$I)] the resolution of Calabi--Yau four-orbifolds with 
finitely many singular points, and an anti-holomorphic 
involution fixing only the singular points \cite{MR1776092}, 
(\cite{MR1787733}).  
\end{enumerate}

$Spin(7)$-instantons are Yang--Mills connections on a 
$Spin(7)$-manifold, which minimize 
the Yang--Mills action. 
They are higher-dimensional analogues of anti-self-dual 
instantons in four dimensions, discussed firstly by physicists such as 
Corrigan--Devchand--Fairlie--Nuyts  
\cite{CoDeFaNu83}, Ward \cite{Ward84}, and later,  
in the String Theory context,  
by Acharya--O'Loughlin--Spence \cite{AcOLSp97},  
Baulieu--Kanno--Singer \cite{BaKaSi98}, and others. 
In mathematics, they were studied 
by Reyes Carri\'{o}n \cite{Reyes_Carrion98}, 
Lewis \cite{Lewis}, Donaldson--Thomas \cite{DT}, 
and later  by 
Donaldson--Segal \cite{DS0902}, 
and several others. 
Analytic results concerning gauge theory in higher dimensions were 
obtained by Nakajima \cite{Nakajima88}, \cite{Nakajima87}, Tian 
\cite{Tian00}, Brendle \cite{Brendle03c}, 
\cite{Brendle03o},  
Tao--Tian \cite{TaTi04}, and others.

Lewis \cite{Lewis} constructed $Spin(7)$-instantons 
on the $Spin(7)$-manifolds of 
the first type (I). 
He constructed them from a family of anti-self-dual instantons on 
$\R^4$ along a Cayley submanifold and glued them together to get a 
$Spin(7)$-instanton on the $Spin(7)$-manifold.

In this article, we construct $Spin(7)$-instantons on the $Spin(7)$ 
manifolds of the second type (I$\!$I). 
The $Spin(7)$-manifold is  obtained by gluing ALE $Spin(7)$-manifolds 
at each singular point of a Calabi--Yau four-orbifold with finitely many 
singular points, and an anti-holomorphic involution fixing only the singular points. 
In this article, assuming that there are Hermitian--Einstein 
connections with certain conditions 
on both the Calabi--Yau four-orbifold and the ALE spaces, 
we glue them together to obtain a $Spin(7)$-instanton on the manifold 
(Theorem 6.1).

The organization of this article is as follows. 
In Section 2, 
we outline Joyce's construction of $Spin(7)$-manifolds from 
a Calabi--Yau four-orbifold with finitely many singular points, and an 
anti-holomorphic involution fixing only the singular set. 
In Section 3, we introduce the $Spin(7)$-instanton 
equation and describe some of its properties, such as 
its relation to the complex ASD equation and the Hermitian--Einstein 
equation, and 
the linearization of them. 
In  Section 4,  we construct approximate solutions from 
Hermitian--Einstein connections with certain conditions on a 
Calabi--Yau 
four-orbifold and ALE spaces,  
and derive an estimate which we need for 
the construction. 
In Section 5, we discuss the linearization of the $Spin(7)$-instanton 
equation, and derive estimates coming from the Fredholm property of the 
linearized operator.  
In Section 6, we give a construction of $Spin(7)$-instantons 
by using the estimates in Sections 4 and 5.  
A simple example of the construction is given in Section 7.

\paragraph{Notations.}Throughout this article, $C$ 
is a positive constant independent of $t$, 
where $t$ is a gluing parameter which is introduced in Section 2.3, 
but it can be different each time it occurs.

\paragraph{Acknowledgements.}

I would like to thank Dominic Joyce for teaching me about the $Spin(7)$-instantons 
and all the other things in this article. 
His ideas and insight can be found throughout this article.  
I am very grateful for his enormous help   
during the composition of this article, including reading the manuscripts so many times. 
I would also like to thank Tommaso Pacini and Heinrich Hartmann for useful conversations, and to thank the Mathematical Institute, Oxford for hospitality. 
This work was supported by a Marie-Curie fellowship of the European Commission under contract number PIIF-GA-2009-235231.

\section{Joyce's second construction of compact $\mathbf{Spin(7)}$-manifolds}
\label{sec:ssm}

We briefly describe the $Spin(7)$-manifolds constructed 
by Joyce in \cite{MR1776092} (see also \cite{MR1787733}, Chapter 15). 
General references for $Spin(7)$-manifolds are Salamon \cite{MR1004008}  
and Joyce \cite{MR1787733}.

\subsection{$\mathbf{Spin(7)}$-manifolds}

The group $Spin(7)\subset SO(8)$ is a compact, connected, 
simply-connected, semi-simple Lie group of dimension 21, 
the double cover of $SO(7)$. 
We introduce it as a subgroup of $GL(8, \R)$ in the following manner.

Let $(x_1, x_2, \dots , x_8)$ be coordinates of $\R^8$, $g_{0}$ the 
standard metric on $\R^8$. 
The $GL(8, \R)$-stabilizer of the 
four-form defined by 
\begin{equation*}
\begin{split}
 \Omega_0 &:= 
\tau^{1256} +\tau^{1278} +\tau^{3456} + \tau^{3478} + \tau^{1357} 
- \tau^{1368} -\tau^{1457} \\
& \qquad + \tau^{2468} - \tau^{1458} -\tau^{1467} - \tau^{2358} 
- \tau^{2367}  + \tau^{1234} + \tau^{5678} ,\\ 
\end{split}
\end{equation*}
where $\tau^{ijkl}$ denotes $dx^i \wedge dx^j \wedge dx^k \wedge dx^l$, 
is isomorphic to the group $Spin(7)$ (\cite{HaLa82}).

The group $Spin(7)$ preserves the metric $g_0$ and an orientation on 
$\R^8$. 
Let $\Omega$ be a four-form on $M$ and $g$ a metric on $M$. 
We call a pair $(\Omega , g)$ a {\it $Spin(7)$-structure} 
if $(\Omega , g)$ 
is isomorphic to  $(\Omega_0 , g_0)$ at each point in $M$. 
We call $\nabla \Omega$ the {\it torsion} of the $Spin(7)$-structure, 
where $\nabla$ is the Levi-Civita connection of $g$, and 
$(\Omega ,g)$ {\it torsion-free} if $\nabla \Omega =0$.

\begin{proposition}[\cite{MR1787733} Proposition 10.5.3] 
Let $M$ be an eight-manifold with a $Spin(7)$-structure $(\Omega , g)$.  
Then the following are equivalent: 
\begin{enumerate} 
\item[$(i)$] $\text{Hol} (g) \subset$ $Spin(7)$;  
\item[$(ii)$] $\nabla \Omega =0$; and  
\item[$(iii)$] $d \Omega =0$. 
\end{enumerate}
\end{proposition}

An eight-manifold with a $Spin(7)$-structure $(\Omega , g)$ is called a 
{\it $Spin(7)$-manifold} if the $Spin(7)$-structure is torsion-free. 
If $g$ has holonomy $\text{Hol} (g) \subset$ $Spin(7)$, then $g$ is 
Ricci-flat. 
The following holds for compact eight-manifolds with holonomy $Spin(7)$.

\begin{theorem}[\cite{MR1787733} Theorem 10.6.8] 
Let $M$ be a compact $Spin(7)$-manifold with torsion-free 
$Spin(7)$-structure $(\Omega , g)$. 
Then $\text{Hol}(g) = \text{$Spin(7)$}$ if and only 
if $\pi_{1} (M) = 0 $ and $\hat{A} (M) =1$. 
\end{theorem}

\subsection{Ingredients for the construction}

In the construction \cite{MR1776092}, $Spin(7)$-manifolds are
constructed 
from the following ingredients: 
\begin{enumerate} 
\item[(A)] a Calabi--Yau four-orbifold $Y$ with only 
isolated singular points, and an anti-holomorphic involution $\sigma$ on
	   $Y$ which fixes only the singular points,   
\item[(B)] ALE $Spin(7)$-manifolds $X_1 , X_2$.  
\end{enumerate} 
We describe each of pieces needed for the construction 
in more detail below.

\paragraph{(A) The Calabi--Yau four-orbifolds. }

A Calabi--Yau $m$-orbifold is a K\"{a}hler orbifold $Y$ of dimension $m$ 
with a K\"{a}hler metric of 
holonomy contained in $SU(m)$.

For the construction, we take a Calabi--Yau four-orbifold $Y$ 
with K\"{a}hler form $\omega$ and holomorphic volume $\theta$. 
We assume that $Y$ has 
finitely many singular points $p_1 ,
p_2 , \dots , p_{k}$ satisfying the following conditions:

\begin{itemize} 
\item 
Each singularity is modeled on $\C^4 / \langle \alpha
\rangle$, where 
$\alpha : \C^4 \to \C^4$  is defined by 
\begin{equation*}
\alpha : (z_1 , z_2 , z_3 ,z_4) \mapsto 
 (i z_1 , i z_2 , i z_3 , i z_4). 
\end{equation*} 
Here $\langle \alpha \rangle \equiv \Z_4$, and $\C^4 / \langle \alpha 
\rangle$ 
has an isolated 
singularity at the origin. 
\item $Y$ has an anti-holomorphic involution $\sigma$, which fixes only  
the singular points $p_1, p_2 , \dots , p_k$. 
\item $Y \setminus \{ p_1 , p_2, \dots , p_{k} \}$ is 
simply-connected, and $h^{2,0} (Y) =0$. 
\end{itemize}

Since $SU(4) \subset Spin(7)$, and $Y$ has holonomy $SU(4)$, there 
is a torsion-free $Spin(7)$-structure on $Y$, which is given by $\Omega 
= \frac{1}{2} \omega^2 + \text{Re} (\theta)$. If we take a 
$\sigma$-invariant $Spin(7)$-structure $(\Omega , g)$ on $Y$, then this
descends to $Z = Y / \langle \sigma \rangle$. 
Hence $Z$ is a $Spin(7)$-orbifold with finitely 
many singular points $p_1 , \dots ,p_{k}$.

For each $j= 1,2, \dots , k$, the tangent space $T_{p_j} Y$ can be 
identified with $\C^{4} / \langle \alpha \rangle$ so that $g_Y$ is 
identified with $|dz_1|^2 + \cdots + |dz_{4} |^2$, $\theta_{Y}$ is 
identified with $dz_1 \wedge \cdots \wedge dz_{4}$, and 
$d \sigma : T_{p_j} Y \to T_{p_j} Y$ is identified with 
$\beta : \C^4 / \langle \alpha \rangle  \to \C^4 / \langle \alpha 
\rangle $ defined by 
\begin{equation*} 
 \beta : (z_1 , z_2 , z_3 ,z_4) \mapsto 
 ( \bar{z}_2 , - \bar{z}_1 , \bar{z}_4 , -\bar{z}_{3}).
\end{equation*}  
Thus, all singularities are modeled on $\R^{8} / \Gamma_{8}$, where $\Gamma_{8} =
\langle \alpha , \beta \rangle$ is a non-abelian group of order 8, 
and there is an isomorphism $i_j : \R^8 / \Gamma_{8} \to T_{p_j} Z$ 
which identifies the $Spin(7)$-structure $(\Omega_{0} ,g_0 )$ on $\R^8 / \Gamma_{8}$ with 
$(\Omega_Z , g_Z)$ on $T_{p_j} Z$ for each $j=1, \dots , k$.

Many examples of Calabi--Yau four-orbifolds satisfying the requirements
in the construction are in hypersurfaces or complete intersections 
in the weighted projective spaces. 
The simplest is the following:

\begin{example}[\cite{MR1787733} pp. $\!$406-407]
Consider the following  hypersurface of degree 12 
in the weighted projective space $\C \mathbb{P}^{5}_{1,1,1,1, 4,4}$, 
given by 
\begin{equation*} 
 [z_{0} :z_1 : z_2 : z_3 : z_4 :z_5 ] 
\in \C \mathbb{P}^{5}_{1,1,1,1, 4,4} : \,  
z_{0}^{12} + z_{1}^{12} + z_{2}^{12} + z_{3}^{12} 
 +z_{4}^{3} +z_{5}^3 = 0 .
\end{equation*}  
Then $c_1 (Y) =0$, thus $Y$ is a Calabi--Yau four-orbifold. 
It has three singular points $p_1 =[0,0,0,0, 1, -1], 
p_2 =[0,0,0,0, 1, e^{i \pi / 3}], 
p_3 =[0,0,0,0, 1, e^{- i \pi / 3}]$.

Define $\sigma : Y \to Y$ by 
\begin{equation*}
  [z_{0} : z_1 : z_2 : z_3 : z_4 :z_5 ] 
 \mapsto  [ \bar{z}_1 : - \bar{z}_0  : \bar{z}_{3} : - \bar{z}_{2} 
: \bar{z}_{5} :  \bar{z}_{4} ] . 
\end{equation*} 
Then $\sigma$ is an anti-holomorphic involution which fixes only the 
singular points $p_1 , p_2 , p_3$.  
\end{example}

\paragraph{(B) The ALE $\mathbf{Spin(7)}$-manifolds. }

Let $\Gamma$ be a finite subgroup of the group $Spin(7)$ 
which acts freely on
$\R^8 \setminus 0$. We call a $Spin(7)$-manifold $M$ with 
$Spin(7)$-structure $(\Omega ,g)$ 
an {\it ALE $Spin(7)$-manifold asymptotic to} $\R^8 / \Gamma$ if there 
is a proper 
continuous surjective map $\pi : X \to \R^8 / \Gamma$ such that 
$\pi : X \setminus \pi^{-1} (0) \to (\R^8 / \Gamma) \setminus 0$ is a 
diffeomorphism, and 
\begin{equation*}
 \nabla^{k} (\pi_{*} (g) -g ) = O ( r^{-8-k}) , 
\quad  \nabla^{k} ( \pi_{*} (\Omega) - \Omega_{0} ) = O (r^{-8-k}) 
\end{equation*} 
on $\{ x \in \R^8 / \Gamma : r(x) > 1 \}$ 
for all $k \geq 0$, where $r$ is the 
radius function on $\R^8 / \Gamma$.

We introduce two types of ALE $Spin(7)$-manifolds denoted by $X_1, X_2$ 
for the construction as follows. 
\begin{enumerate}
 \item[(I)] Define complex coordinates $(z_1, z_2 , z_3 , z_4)$ on $\R^8$ by 
\begin{equation*}
 (z_1, z_2 , z_3 , z_4) = (x_1 + i x_2, 
x_3 + i x_4 , x_5 + i x_6 , x_7 + i x_8) . 
\end{equation*} 
Then $g_{0} = |dz_{1}|^2 + \cdots + |dz_{4}|^2$ and 
$\Omega_{0} = \frac{1}{2} \omega_{0} \wedge \omega_{0} + 
\text{Re} (\theta_{0})$.

Define $\alpha , \, \beta : \C^4 \to \C^4$ by 
\begin{gather*}
 \alpha : (z_1 , z_2 , z_3 ,z_4) \mapsto 
 (i z_1 , i z_2 , i z_3 , i z_4), \\ 
 \beta : (z_1 , z_2 , z_3 ,z_4) \mapsto 
 ( \bar{z}_2 , - \bar{z}_1 , \bar{z}_4 , -\bar{z}_{3}).  
\end{gather*}
We denote by $W_{1}$ the crepant resolution 
$\pi_1 : W_1 \to \C^4 / \langle \alpha \rangle$ of $\C^4 / \langle \alpha \rangle$, 
which is the blow-up of $\C^4 / \langle \alpha \rangle$ at $0$, 
with $\pi_{1}^{-1} (0) = \C \mathbb{P}^3 $. 
The action of $\beta$ lifts to a free anti-holomorphic involution of $W_1$. 
Hence $X_{1} = W_{1} / \langle \beta \rangle$ is a resolution of $\R^8 / \Gamma_{8}$, 
where $\Gamma_{8} = \langle \alpha , \beta \rangle$.

\item[(I$\!$I)] There is another complex structure on $\R^8$, namely, we define 
complex coordinates $(w_1, w_2 , w_3 , w_4)$ on $\R^8$ by 
\begin{equation*} 
 (w_1, w_2 , w_3 , w_4) = (- x_1 + i x_3, 
x_2 + i x_4 , - x_5 + i x_7 , x_6 + i x_8) . 
\end{equation*} 
Then $g_{0} = |dw_{1}|^2 + \cdots + |dw_{4}|^2 $ and 
$\Omega_{0} = \frac{1}{2} \omega_{0}' \wedge \omega_{0}' + 
\text{Re} (\theta_{0}')$. 
Define $\alpha , \, \beta : \C^4 \to \C^4$ by 
\begin{gather*} 
\alpha : (w_1 , w_2 , w_3 ,w_4) \mapsto 
 (\bar{w}_2 , - \bar{w}_1 ,  \bar{w}_4 , - \bar{w}_3), \\ 
 \beta : (w_1 , w_2 , w_3 ,w_4) \mapsto 
 (i w_1 , i w_2 , i w_3 , i w_4) . 
\end{gather*}
We denote by $W_{2}$ the crepant resolution 
$\pi_2 : W_2 \to \C^4 / \langle \beta \rangle$ of $\C^4 / \langle \beta \rangle$, 
which is the blow-up of $\C^4 / \langle \beta \rangle$ at $0$, 
with $\pi_{2}^{-1} (0) = \C \mathbb{P}^3 $. 
The action of $\alpha$ lifts to a free anti-holomorphic 
involution of $W_2$. 
Hence $X_{2} = W_{2} / \langle \alpha \rangle$ is a resolution of 
$\R^8 / \Gamma_{8}$, 
where $\Gamma_{8} = \langle \alpha , \beta \rangle$.  
\end{enumerate}

\subsection{The manifolds}

We glue either $X_{1}$ or $X_{2}$ to each singular point $p_{j} \, (j= 
1, \dots ,k)$ of the $Spin(7)$-orbifold  $Z$ to obtain a compact smooth 8-manifold $M$.

Firstly, we have the following description around the singular points 
of $Z$:   
We denote by $\exp_{p_j} : T_{p_j} Z \to  Z$ the exponential map. 
Then  $\exp_{p_j} \circ \, i_j$ maps $\R^8 / \Gamma_{8}$ to $Z$. 
We take $\zeta$ small and define $U_j \subset Z$ by $U_{j} = \exp_{p_j} 
\circ \, i_j (B_{\zeta} (\R^8 / \Gamma_{8}))$, 
where $B_{\zeta}(\R^8 / \Gamma_{8})$ is the 
open ball of radius $ \zeta$ about $0$. 
We take $\zeta$ small enough so that $U_{j}$ is open in $Z$ and 
$\psi_{j} := \exp_{p_j} \circ \, i_j : B_{\zeta} (\R^8 / \Gamma_{8}) \to
U_{j}$ is 
a diffeomorphism for $j = 1, \dots ,k$, and that $U_{i} \cap U_{j} =
\empty$ for $i \neq j$.

Next, we introduce a scaling parameter $t \in ( 0 , 1]$. 
For each $i=1,2$  
we consider the rescaled ALE $Spin(7)$-manifold $X_{i}^{t} =X_{i}$ with
a $Spin(7)$-structure $(\Omega_{i}^{t} ,g_{i}^{t})$ defined by 
\begin{equation*}
 \Omega_{i}^{t} = t^4 \Omega_{i} , \quad g_{i}^{t} = t^2 g_{i} ,
\end{equation*}
and the projection $\pi_{i}^{t} : X_{i}^{t} \to \R^8 / \Gamma_{8}$ 
given by $\pi_{i}^{t} = t \pi_{i}$. 
Then each 
$(X_{i}^{t} , \Omega_{i}^{t} ,
g_{i}^{t}) \, (i=1,2)$  is an ALE $Spin(7)$-manifold asymptotic to $\R^8 
/ \Gamma_{8}$.

We now 
define  $M_{i}^{t} \, (i= 0,1,2, \dots, k)$ by 
\begin{gather*} 
M_{0}^{t} = Z \setminus \bigcup_{j=1}^{k} \psi_{j} 
( \overline{B_{t^{\frac{5}{6}} \zeta}} 
(\R^8 / \Gamma_{8})) \subset Z , \\ 
 M_{j}^{t} =  (\pi_{n_j}^{t})^{-1} ( B_{t^{\frac{3}{4}} \zeta} (\R^8 /
 \Gamma_{8})) \subset X_{n_{j}} 
\qquad (j=1, 2, \dots ,k), 
\end{gather*} 
where $n_j = 1$ or $2$ for $j=1,2, \dots , k$. 
Now we define a resolution $M= M^{t}$ of $Z$ by 
$\coprod_{j=0}^k M^{t}_{j} /\sim$, where the equivalence relation 
$\sim$ is given by $x \sim y$ if either (a) $x=y$, 
\begin{enumerate} 
\item[(b)] $x \in M_{j}^{t}$ and $y \in U_{j} \cap M_{0}^{t}$, and 
$\psi \circ \pi_{n_j}^{t} (x) = y$ for $j=1, \dots, k$, or  
\item[(c)] $y \in M_{j}^{t}$ and $x \in U_{j} \cap M_{0}^{t}$, and 
$\psi \circ \pi_{n_j}^{t} (y) = x$ for $j=1, \dots, k$. 
\end{enumerate}
Then $M$ is a compact 8-manifold, and $\pi_1 (M) = \Z_{2}$ if 
$n_{j} =1$ for all $j = 1, \dots , k$, or otherwise, i.e.,  if $n_j =2$
for some $j$, $M$ is 
simply-connected.

We define a {\it radius function} on $M^{t}$ as follows: 
Firstly, define a radius function on 
$M_{0}^{t} \subset Z$ to be a function
$\rho_{M_{0}^{t}} :  M_{0}^{t} \to [t^{\frac{5}{6}} , 1]$ such 
that 
$\rho_{M_{0}^{t}} \circ \pi_{j} 
= r $ for $r$ in $[ t^{\frac{5}{6}} ,  \zeta] $ 
and $j = 1, \dots , k $, 
and $\rho_{M_{0}^{t}}$ in $[\zeta , 1]$ on $Z \setminus
\bigcup_{j=1}^{k} U_{j}$. 
Secondly, define a radius function on $M_{j}^{t} \subset 
X_{n_{j}} \, 
(j= 1, \dots , k)$ to be a function 
$\rho_{M_{j}^{t}} : M_{j}^{t} \to [t , t^{\frac{3}{4}}]$ such that 
$\rho_{M_{j}^{t}} = t ( r \circ \pi_{j}^{t})$ for $r$ in $[1 , t^{-
\frac{1}{4}}]$ 
and $t$ for $r <1$. 
Then these functions $\rho_{M_{0}^{t}}, \,  \rho_{M_{j}^{t}}$ 
coincide on each $M_{0}^{t} \cap M_{j}^{t} \, (j = 1 , \dots , k)$, 
hence we get a radius function $\rho : M^{t} \to [t , 1]$ from these.

\subsection{The $\mathbf{Spin(7)}$-structure}

We glue together the torsion-free $Spin(7)$-structures  
$(\Omega_Z ,g_Z)$ on $M_{0}^{t}$ and $(\Omega_{n_{j}}^{t} ,g_{n_j}^{t})$
on $M_{j}^{t} \, (j= 1, \dots ,k)$ by a partition of unity.

Firstly, under the identification of $B_{\zeta} (\R^8 / \Gamma_{8}) $
with $U_{j} \subset Z$ by $\psi_{j}$, there is a smooth three-form 
$\sigma_{j}$ on 
$B_{\zeta} (\R^8 / \Gamma_{8})$ such that 
$\psi_{j}^{*} (\Omega_{Z}) - \Omega_0 
= d \sigma_{j}$ for each $j= 1, \dots ,k$ with $| \nabla^{\ell} 
\sigma_{j} | \leq D_1 r^{3 - \ell} \, ( \ell = 0,1,2)$ on $B_{\zeta} 
(\R^8 / \Gamma_{8})$, where $D_1 >0$ is a constant independent of $t$, 
and $|\cdot| ,\, \nabla$ are 
with respect to the metric $g_{0}$ (\cite{MR1787733} Proposition 15.2.6).

On the other hand, 
there exists a smooth three-form $\tau_{j}^{t}$ on 
$(\R^8 / \Gamma_{8}) \setminus B_{t \zeta} (\R^8 / \Gamma_{8})$ 
such that 
$(\pi_{n}^{t})_{*} (\Omega_{n}^{t}) = \Omega_0 + d \tau_{n}^{t}$ with 
$|\nabla^{\ell} \tau_{n}^{t} | \leq D_2 t^8 r^{-7 - \ell}$ for 
$\ell = 0, 1, 2$ on $(\R^8 / \Gamma_{8}) 
\setminus B_{t \zeta} (\R^8 / \Gamma_{8})$, 
where $D_2 >0$ is a constant independent of $t$, 
and $|\cdot| ,\, \nabla$ are with respect to 
the metric $g_{0}$. This can be proved by using an 
explicit metric by Calabi \cite{MR543218}.

Let $\eta :[ 0 , \infty ) \to [0 ,1]$ be a smooth function with 
$\eta(x) =0$ for $x \leq 1$ and $\eta (x) =1$ for $x \geq 2$. 
We take $t$ small enough that $2 t^{\frac{4}{5}} \leq t^{\frac{3}{4}} $, and 
define a closed four-form $\xi^{t}$ on $M^{t}$ by $\xi^{t} = \Omega_{Z}$ on 
$M_{0}^{t} \setminus \bigcup_{j=1}^{k} M_{j}^{t}$, and 
$\xi^{t} = \Omega_{n_j}^{t}$ 
on $M_{j}^{t} \setminus M_{0}^{t} \, (j = 1 ,\dots ,k)$, and 
\begin{equation*}
 \xi^t = \Omega_{0} + d (\eta (t^{- \frac{4}{5}}r ) \sigma_{j} ) 
+ d ( (1-\eta(t^{- \frac{4}{5}} r)) \tau_{n_j}^{t} ) 
\end{equation*} 
on $M_{0}^{t} \cap M_{j}^{t} \, (j = 1 , \dots ,k)$. 
Here we identify $M_{0}^{t} \cap M_{j}^{t}$ with an annulus in 
$\R^{8} / \Gamma_{8}$.

By using this $\xi^{t}$, we can construct a family of $Spin(7)$-structures 
$(\Omega^{t} , g^{t})$ for small $t$ and the difference $\phi^{t} = 
\xi^{t} -\Omega^t$ can be estimated by 
\begin{equation*} 
  || \phi^{t} ||_{L^2} \leq \lambda t^{\frac{13}{3}} , \quad 
   || d \phi^{t} ||_{L^{10}} \leq \lambda t^{\frac{7}{5}}, 
\end{equation*}  
where $\lambda$ is a constant, as well as 
$ \delta (g^{t}) \geq \mu t$ and $|| R(g^{t}) ||_{C^{0}} \leq \nu t^{-2}$, 
where $\delta (g^{t})$ is the injective radius of $g^{t}$, $R(g^{t})$ is the 
Riemannian curvature, and $\mu, \nu >0$ are constants 
(\cite{MR1787733} Theorem 15.2.13). 
Here all norms are calculated by the metric $g^{t}$ on $M^t$.

Then the existence of torsion-free $Spin(7)$-structures follows from  
\begin{theorem}[\cite{MR1787733} Theorem 13.6.1, Proposition 13.7.1]
Let $\lambda$, $\mu$, $\nu >0$ be constants. Then there exists 
constants $\kappa, K >0$ such that for $0 < t \leq \kappa$ the 
following holds. 
Let $M$ be a compact 8-manifold, and $(\Omega^{t} , g^{t})$ a
 $Spin(7)$-structure 
on $M$. 
Suppose that $\phi^{t}$ is a four-form on $M$ 
with $d \Omega^{t} + d \phi^{t} = 0$, and 
\begin{enumerate}
\item[(i)] $ || \phi^{t} ||_{L^2} \leq \lambda t^{\frac{13}{3}}$ 
   and $|| d \phi^{t} ||_{L^{10}} \leq \lambda t^{\frac{7}{5}}$;  
\item[(ii)] the injectivity radius 
$\delta (g^{t})$ satisfies $ \delta (g^{t}) \geq \mu t$; 
and 
\item[(iii)] the Riemannian curvature $R(g^{t})$ 
satisfies $|| R(g^{t}) ||_{C^{0}} \leq \nu t^{-2}. 
$\end{enumerate} 
Then there exists a smooth torsion-free $Spin(7)$-structure 
$( \tilde{\Omega}^{t} , \tilde{g}^{t})$ on $M$ such that 
$|| \tilde{\Omega}^{t} -\Omega^{t} ||_{C^{0}} \leq K t^{\frac{1}{3}}$ and 
$|| \nabla ( \tilde{\Omega}^{t} - \Omega^{t} ) ||_{L^{10}} \leq K t^{\frac{2}{15}}$.  
\label{th:spinss}
\end{theorem}

By this theorem, we can deform the $Spin(7)$-structure 
$(\Omega^{t}, g^{t})$ above to a torsion-free $Spin(7)$-structure 
$( \tilde{\Omega}^{t} , \tilde{g}^{t})$ on $M$ for $t$ sufficiently small. 
Theorem 2.2 then shows that $\text{Hol} (\tilde{g}^t) = Spin(7)$
provided $\pi_1 (M^{t}) =0$ and $\hat{A} (M^t) =1$.

\section{$\mathbf{Spin(7)}$-instantons}

In Section 3.1, we introduce the $Spin(7)$-instanton equation, and
describe its relation to the complex ASD equation 
and the Hermitian--Einstein equation. 
In Section 3.2, we describe the linearization of the $Spin(7)$-instanton 
equation and the Hermitian--Einstein equation.

\subsection{$\mathbf{Spin(7)}$-instantons}

Let $M$ be a $Spin(7)$-manifold. 
Then the space of two-forms $\Lambda^2 $ on $M$ splits as 
\begin{equation*} 
 \Lambda^2 = \Lambda^2_{21} \oplus \Lambda^2_{7} , 
\end{equation*} 
where $\Lambda^{2}_{21}$ is a rank 21 vector bundle which  
corresponds to the Lie algebra of $Spin(7)$ under 
the identification of $\Lambda^2$ with the Lie algebra of $SO(8)$,  
and $\Lambda^{2}_{7}$ is a rank 7 vector bundle which is 
orthogonal to $\Lambda^{2}_{21}$. 
Alternatively, if we consider the operator on $\Lambda^2$ defined by 
$\alpha \mapsto * (\Omega \wedge \alpha)$, then it is self-adjoint with 
eigenvalues $-1$ and $3$, and its eigenspaces are $\Lambda_{21}^{2}$  
and $\Lambda_{7}^{2}$ respectively.

Let $P$ be a principal bundle on $M$ with the structure group $G$.  
We denote by $\text{Ad}(P)$ the adjoint vector bundle associated with $P$. 
The space of $\text{Ad}(P)$-valued 
2-forms is also decomposed as 
\begin{equation*} 
 \Omega^2 (\text{Ad}(P))= \Omega^2_{21} (\text{Ad}(P)) \oplus \Omega^2_{7} 
 (\text{Ad}(P)). 
\end{equation*}

We call a connection $A$ on $P$ a {\it $Spin(7)$-instanton} if $A$
satisfies the following equation:   
\begin{equation} 
 \pi_{7}^{2} (F_{A}) =0 , 
\label{eq:ssinst}
\end{equation}
where $F_{A}$ is the curvature of $A$, and 
$\pi_{7}^{2}$ is the projection to the $\Omega^2_7 
(\text{Ad}(P))$ component. 
Equation \eqref{eq:ssinst} together with a gauge fixing condition 
form an elliptic system. 
Note that the projection $\pi_{7}^{2} : \Omega^{2} (\text{Ad} (P)) \to 
\Omega_{7}^{2} (\text{Ad} (P))$ can be written as 
\begin{equation} 
\alpha \mapsto 
\frac{1}{4} (* \Omega \wedge \alpha + \alpha ) 
\label{projomega} 
\end{equation} 
for $\alpha \in \Omega^{2} (\text{Ad} (P))$.

\paragraph{Complex ASD. }

Let $M$ be a compact Calabi--Yau four-fold with K\"{a}hler form 
$\omega$ 
and holomorphic $(4,0)$-form $\theta$. 
We assume the normalization condition  
$ \theta \wedge \bar{\theta} = \frac{16}{4!} \omega^4$  on $\omega$ and $\theta$. 
A Calabi--Yau four-fold is a $Spin(7)$-manifold as $SU(4) \subset 
Spin(7)$, and the $Spin(7)$-structure $\Omega$ is given by $\Omega = 
\frac{1}{2} \omega^2 + \text{Re} (\theta)$. 
Let $E$ be a Hermitian vector bundle over $M$.

In general, if the underlying manifold is K\"{a}hler, then 
we have the following decomposition of the space of complexified two 
forms:  
\begin{equation*} 
 \Lambda^2 \otimes \C = \Lambda^{1,1} \oplus \Lambda^{2,0} 
\oplus \Lambda^{0,2}, 
\end{equation*}
and $\Lambda^{1,1}$ further decomposes into $ \C \langle \omega \rangle \oplus \Lambda_{0}^{1,1}$.

In the case where the underlying K\"{a}hler manifold $M$ is a
Calabi--Yau four-fold,  
we define {\it the complex Hodge operator} 
$ *_{\theta} : \Lambda^{0,2} \to \Lambda^{0,2} $ 
by 
$$  \phi  \wedge *_{\theta} \psi 
 = \langle \phi , \psi \rangle  \bar{\theta} , \qquad \phi , \psi \in 
 \Lambda^{0,2} . $$ 
Then $*_{\theta}^2 =1$, and the space of 
$(0,2)$-forms further decomposes into  
\begin{equation*}
 \Lambda^{0,2} = \Lambda^{0,2}_{+} \oplus \Lambda_{-}^{0,2}, 
\end{equation*}
where 
\begin{gather*}
 \Lambda^{0,2}_{+} 
= \{ \phi \in \Lambda^{0,2} \, : \, *_{\theta} \phi = \phi \} , \\
 \Lambda^{0,2}_{-} 
= \{ \phi \in \Lambda^{0,2} \, : \, *_{\theta} \phi = - \phi \} . 
\end{gather*}
Note that the operator $*_{\theta}$ is an anti-holomorphic map, hence 
$\Lambda^{0,2}_{+}$ and $\Lambda^{0,2}_{-}$ are real subspaces of 
$\Lambda^{0,2}$. 
We obtain 
\begin{gather*}
 \Lambda^{2}_{21} = \Lambda^2 \cap ( \Lambda_{0}^{1,1}  
\oplus 
\Lambda^{0,2}_{-} \oplus \Lambda^{2,0}_{-}), \\
 \Lambda^{2}_{7} = \R \langle \omega \rangle 
\oplus (  \Lambda^{2} \cap (\Lambda^{0,2}_{+} \oplus 
 \Lambda^{2,0}_{+} ) ). 
\end{gather*} 
Hence, the $Spin(7)$-instanton equation on a Calabi--Yau four-fold can
be written as 
\begin{equation*}
(1 + *_{\theta}) F_{A}^{0,2} = 0 , \quad 
\Lambda F_{A}^{1,1}  = 0. 
\end{equation*}
These are called \textit{complex 
anti-self-dual equations} \cite{DT}.

\paragraph{Hermitian--Einstein connections.}

Hermitian--Einstein connections also give examples of 
$Spin(7)$-instantons.

Let $X$ be a compact K\"{a}hler manifold of complex dimension $n$ with 
K\"{a}hler form $\omega$, and $E$ a Hermitian vector bundle over $X$ 
with Hermitian metric $h$.

A metric-preserving connection $A$ of $E$ is said to be a 
\textit{Hermitian--Einstein connection} if $A$ satisfies the following 
equations: 
\begin{equation*} 
 F_{A}^{0,2} =0, \quad 
\Lambda F_{A}^{1 ,1}   = \lambda(E) Id_{E}, 
\end{equation*}
where $\Lambda := (\wedge \omega)^{*}$, and 
$\lambda (E)$ is defined by 
$ \lambda (E) := 
\frac{n ( c_1(E) \cdot [\omega]^{n-1} )}{r [\omega]^{n}}$.

If $E$ is a unitary vector bundle with $c_1 (E) \cdot [\omega]^3 =0$ 
over a K\"{a}hler manifold with 
holonomy contained in $SU(4)$, then $\lambda (E) =0$ and $F_{A} 
\in \Lambda_{0}^{1,1}$, and  Hermitian--Einstein connections 
are $Spin(7)$-instantons, since $\Omega^2 \cap \Omega_{0}^{1,1} 
\subset \Omega^{2}_{21}$ as described above.

\subsection{Linearizations}

The infinitesimal deformation of $Spin(7)$-instantons was studied by 
Reyes Carri\'{o}n \cite{Reyes_Carrion98}, and it is given 
by the following 3-term complex: 
\begin{equation} 
0 \longrightarrow \Omega^{0} (\mathfrak{u} (E)) 
\xrightarrow{ \, \, \, \, \, d_{A}  \, \, \, \, \,} 
\Omega^{1} (\mathfrak{u} (E)) 
\xrightarrow{ \, \, \, \, \, d_{A}^{7}  \, \, \, \, \, } 
\Omega^{2}_{7} (\mathfrak{u} (E)) 
\longrightarrow 0, 
\label{eq:complex} 
\end{equation} 
where $d_{A}^{7} := \pi^{2}_{7} \circ d_{A}$. 
This complex is elliptic \cite{Reyes_Carrion98}, hence 
\begin{equation} 
 L_{A} := (d_{A}^{7} , d_{A}^{*}) : \Omega^{1} (\mathfrak{u} (E)) \to 
\Omega^{0} (\mathfrak{u} (E)) \oplus \Omega^{2}_{7} (\mathfrak{u} (E)) 
\label{eq:lineop}
\end{equation}
is an elliptic operator. 
The local model of the moduli space of $Spin(7)$-instantons 
is described in Lewis' thesis \cite{Lewis}.  
The operator $L_{A}$ is the twisted Dirac operator 
between the Spin bundles twisted by 
$\mathfrak{u} (E)$: 
\begin{equation}
 S^{+}= \Omega^0 (\mathfrak{u} (E)) \oplus \Omega^{2}_{7}  (\mathfrak{u} (E)), 
\quad S^{-} = \Omega^1  (\mathfrak{u} (E)).  
\label{eq:spinor}
\end{equation}
Hence, the index of the complex \eqref{eq:complex} can be calculated 
by the Atiyah--Singer Index Theorem, it is $\langle \hat{A}
(M)  \, \text{ch}  (\mathfrak{u}  (E)), [M] \rangle$, and 
when $M$ is a compact 8-manifold with holonomy $Spin(7)$, 
$\text{Ind} (L_{A})$ turns out to be 
\begin{equation}
 \begin{split}
 \text{Ind} &(L_{A}) 
 = - r^2 - \left\langle - \frac{ p_1 (M)}{24} 
\left( -c_1 (E)^2 + r \left( c_{1} (E)^2 - 2 c_2 (E) \right) \right) 
\right. \\ 
 &+ \frac{r}{12} \left( 
 c_1 (E)^4 - 4 c_1 (E)^2 c_2 (E) + 2 c_2 (E)^2 + 4 c_1 (E) c_3 (E) 
 - 4 c_4 (E) \right) \\
 & \qquad \qquad \qquad  \left. - \frac{1}{12} c_1 (E)^4 - c_1 (E) c_3 (E) 
 + c_2 (E)^2 , \, [M] \right\rangle , \\ 
 \end{split}
\label{eq:index}
\end{equation}
where we used the fact that 
$ \left\langle -4 p_{2}(M) + 7 p_1 (M)^2  
 , \, [M] \right\rangle = 5760$ if $M$ has holonomy  $Spin(7)$. 

If $E$ is an $SU(r)$ bundle, rather than a $U(r)$ bundle, 
then we replace $\mathfrak{u}(E)$ by $\mathfrak{su}(E)$ 
in \eqref{eq:complex}, \eqref{eq:lineop}, and \eqref{eq:spinor}, 
and the first term $r^2$ in
\eqref{eq:index} is replaced by $r^2 -1$. 
In particular, 
if we take $E$ to be an $SU(2)$ bundle, then  \eqref{eq:index} becomes  
\begin{equation}
 \text{Ind} (L_{A}) 
= - 3 - \frac{1}{6} 
\left\langle p_1 (M) c_2 (E) + 8 c_2 
 (E)^2 , \, [M] \right\rangle . 
\label{eq:vdim}
\end{equation}

\paragraph{Infinitesimal deformation of Hermitian--Einstein connections. }

Let $X$ be a K\"{a}hler four-fold 
with K\"{a}hler form $\omega$, $E$ a Hermitian vector bundle over $X$ 
with Hermitian metric $h$.

The infinitesimal deformation of a Hermitian--Einstein 
connection $A$ of $E$ 
was studied by Kim \cite{MR888136} (see also \cite{Kobayashi87}), 
and it is described by the following complex:  
\begin{equation} 
\begin{split} 
 0 \longrightarrow \Omega^{0} ( X, &\mathfrak{u}(E)) 
 \xrightarrow{ \, \, \, \, \,  d_{A} \, \, \, \, \,  } 
 \Omega^1 ( X, \mathfrak{u}(E)) 
 \xrightarrow{ \, \, \, \, \, d_{A}^{+}  \, \, \, \, \,} 
 \Omega^{+} (X,  \mathfrak{u}(E))  \\
  & \qquad 
  \xrightarrow{ \, \, \, \, \, \bar{D}_{A}'  \, \, \, \, \,} 
A^{0,3} ( X, \mathfrak{u}(E)) 
 \xrightarrow{  \, \, \, \, \, \bar{D}_{A}  \, \, \, \, \, } 
A^{0,4} (X, \mathfrak{u}(E))  
\longrightarrow 0 \\ 
\end{split}
\label{defHE}
\end{equation}
where 
\begin{equation*} 
  A^{0,q} (X ,\mathfrak{u}(E)) 
:= C^{\infty} (\mathfrak{u} (E) \otimes A^{0,q} ), 
\end{equation*} 
$\mathfrak{u}(E) = \text{End} (E, h)$ is the bundle of 
skew-Hermitian endmorphisms of $E$, 
$A^{0,p}$ is the space of real 
$(0,p)$-forms (see \cite{MR1004008} pp. 32--33) over $X$, defined by 
\begin{equation*}
 A^{0,p} \otimes _{\R} \C =\Lambda^{0,p}  \oplus \Lambda^{p,0} , 
\end{equation*}
\begin{equation*}
 \begin{split}
 \Omega^{+} (X ,\mathfrak{u}(E)) 
&:= A^{0,2} (X ,\mathfrak{u}(E)) 
  \oplus \Omega^0 (X ,\mathfrak{u}(E)) \omega \\
&= \{ \phi + \bar{\phi} + f \omega \, : \, 
\phi \in \Omega^{0,2} (X ,\mathfrak{u}(E)) ,\, f \in \Omega^{0} (X ,\mathfrak{u}(E))  \} ,
 \end{split}
\end{equation*} 
$\bar{D}_{A} : A^{0,p} (X , \mathfrak{u} (E)) \to  A^{0, p+1} (X , 
\mathfrak{u}(E))$ is defined by $\bar{D}_{A} \alpha = \bar{\partial}_{A} 
\alpha^{0,p} + \partial_{A} \overline{\alpha^{0,p}}$ for $\alpha = \alpha^{0,p} +
\overline{\alpha^{0,p}}$, where $\alpha^{0,p} \in \Omega^{0,p} (X ,
\mathfrak{u} (E))$, and 
$$d_{A}^{+} := \pi^{+} \circ d_{A}, \quad \bar{D}_{A} ' := \bar{D}_{A} 
\circ 
\pi^{0,2} ,$$ 
where $\pi^{+} , \pi^{0,2}$ are respectively the orthogonal 
projections from $\Omega^2 $ to 
$\Omega^{+} , A^{0,2}$. 
As described in \cite{MR888136} (see also \cite{Kobayashi87}), 
the complex \eqref{defHE} has the associated Dolbeault complex: 
\begin{equation*}
\begin{CD}
0 @>>> \Omega^0 @>d_{A}>> \Omega^1 @>d_{A}^{+}>>\Omega^{+} @>\bar{D}_{A}'>>  
A^{0,3} @>\bar{D}_{A}>> A^{0,4} @>>> 0 \\ 
@. @VVj_0V @VVj_1V @VVj_2V @VVj_3V @VVj_4V \\ 
0 @>>> \Omega^{0,0} @>\bar{\partial}_{A}>> \Omega^{0,1} 
 @>\bar{\partial}_{A}>>  
\Omega^{0,2}  
@>\bar{\partial}_{A}>> \Omega^{0,3} @>\bar{\partial}_{A}>> 
\Omega^{0,4} @>>> 0 , 
\end{CD} 
\end{equation*}
where $j_{0}$ is injective, $j_1$ is bijective, $j_2$ is surjective with  
the kernel $\{ \beta \omega \, : \,  \beta \in \Omega^{0} \}$, and 
$j_{3}, j_{4}$ are bijective. 
We denote the $i$-th cohomology of the complex \eqref{defHE} by $H^{i} 
\, (i = 0, 1, \dots , 4)$.  
Kim \cite{MR888136} proved that 
\begin{gather*} 
 H^{0} \otimes \C \cong H^{0,0}, \, H^{1} \cong H^{0,1} , \, 
H^{2} \cong H^{0} \oplus H^{0,2} , \\ 
H^{3} \cong  H^{0,3 }  , \, H^{4} \cong H^{0,4} . 
\end{gather*}
In particular, if an $SU(r)$ bundle $E$ is irreducible, and $H^{0,2}=0$,  
then the linearized operator $L_{A}= ( d_{A}^{+} ,d_{A}^{*}) : 
\Omega^{1} (\mathfrak{su} (E)) \to 
\Omega^{0} (\mathfrak{su} (E)) \oplus \Omega^{2}_{7} (\mathfrak{su} (E)) $ is surjective.

\section{Approximate solution and the estimate}

In this section, we construct an approximate solution to the 
$Spin(7)$-instanton equation on a 
vector bundle over the $Spin(7)$-manifold $M=M^t$ of Section 2.  
The ingredients are 
Hermitian--Einstein connections on vector bundles 
over 
the Calabi--Yau four-orbifold $Z$ and the ALE spaces   
$W_{n_{j}}$'s. 
We also prove an estimate on the approximate solution needed 
in the later section.

\subsection{Ingredients for the construction}

We take a complex  
vector bundle $E_{0}^{j}$ of rank $r$ over $(\R^8 / \Gamma_8 ) 
\setminus 0$ with a flat $U(r)$-connection $A_{0}^{j}$ 
for each $j = 1 , \dots , k$. 
Then ingredients for the construction consist of 
\begin{enumerate}
\item[(A)] A complex orbifold vector bundle $E_{Z} = E_{Y} / 
\langle \sigma \rangle$ of rank $r$ 
over $Z= Y / \langle \sigma \rangle$, 
which is isomorphic to $E_{0}^{j}$ near each singular point 
$p_{j} \in Z\, (j= 1 ,2, \dots ,k $), 
where $E_Y$ is a $\langle \sigma \rangle$-equivariant holomorphic  
	   orbifold vector bundle over $Y$, 
equipped with a $\langle \sigma \rangle$-equivariant 
Hermitian--Einstein connection $A_{Y}$,  
and the connection $A_{Z}$ on $E_{Z}$ induced by $A_{Y}$ is 
asymptotic to the flat connection $A_{0}^{j}$ 
with the decay rate 
$A_{Z} \sim  A_{0}^{j} + O (r)$ 
and $\nabla^{\ell} (A_Z - A_{0}^{j}) \sim O (1)$ for all $\ell > 0$  
at each 
$p_{j} \in Z \, (j= 1 ,2, \dots ,k $). 
\item[(B)] A complex vector bundle $E_{X_{n_j}} = 
E_{W_{n_j}} /\langle \sigma \rangle$ 
of rank $r$ over each
	   $X_{n_j}$  $(j= 1, \dots ,k)$, 
which is isomorphic to $E_{0}^{j}$ near $\infty$, 
where $E_{W_{n_j}}$ is a $\langle \sigma \rangle$-equivariant
	   holomorphic 
	   vector bundle over $W_{n_j}$, 
equipped with a 
$\langle \sigma \rangle$-equivariant 
Hermitian--Einstein connection 
$A_{W_{n_j}}$, and the connection $A_{X_{n_j}}$ on $E_{X_{n_j}}$ 
induced by 
$A_{W_{n_j}}$  is  asymptotic to the flat connection $A_{0}^{j}$ 
with the decay rate $ A_{X_{n_j}} \sim  A_{0}^{j} + O(r^{-7})$ 
and $\nabla^{\ell} (A_{X_{n_{j}}} - A_{0}^{j}) \sim O ( r^{-7 - \ell })$
	   for all $\ell > 0 $  at infinity. 
\end{enumerate} 
We also assume that the cokernel of $L_{A_{Z}}$ lies in  
$C^{\infty} (\mathfrak{u}(E_{Z}) \otimes \Lambda^{0} (Z))$, 
namely, the cohomology $H^{2} (Z , \mathfrak{u} (E))$ of the complex 
\eqref{eq:complex} vanishes, but $H^{0} (Z , \mathfrak{u}(E))$ of the
complex 
\eqref{eq:complex} does not necessarily vanish,  
and $ L_{A_{X_{n_j}}} : L_{1,
\delta}^{4} (\mathfrak{u} (E_{X_{n_{j}}}) \otimes \Lambda^{1} (X_{n_{j}}) ) 
\to L_{\delta -1}^{4}  
(\mathfrak{u} (E_{X_{n_{j}}}) \otimes ( \Lambda^{0} (X_{n_{j}}) \oplus \Lambda_{7}^{2} 
(X_{n_{j}}) ))
\, (j= 1, 2, \dots , k)$ 
for $\delta \in ( -7, 0)$ 
is surjective, 
where $L_{1,
\delta}^{4} (\mathfrak{u} (E_{X_{n_{j}}}) \otimes  \Lambda^{1}
(X_{n_{j}}) ) 
$ and $L_{\delta -1}^{4}
(\mathfrak{u} (E_{X_{n_{j}}}) 
\otimes ( \Lambda^{0} (X_{n_{j}}) \oplus \Lambda_{7}^{2} 
(X_{n_{j}}) ))$ are weighted Sobolev spaces with the
weights $\delta, \delta-1$ (See Section 5.1 for more detail 
about the weighted Sobolev spaces). 

Note that we do not assume $E_{Z}$ or $E_{X_{n_j}} \, (j = 1 , \dots ,
k)$ to be irreducible, for instance, $E_{Z}$ can be trivial. 
In fact, even if $E_Z$ and $E_{X_{n_j}}$ are reducible, one can
construct irreducible $Spin(7)$-instantons, provided that the
intersection of the symmetry groups of $E_{Z}, E_{X_{n_j}} \, (j = 1,
\dots , k)$ is the multiples of the identity.  

Also notice that for both $E_Y$ and $E_{W_{n_j}}$, the constant $\lambda
(E)$ in Section 3.1 are zero, since $E_Y$ and $E_{W_{n_{j}}}$ are
assumed to be $\sigma$-equivariant, so $\lambda (E) = \frac{4 \, c_{1} (E)
\cdot [\omega]^3}{r [\omega]^4}$ changes sign under the action of
$\sigma$. 
Therefore $A_{Y}$ and $A_{W_{n_j}}$ are $Spin(7)$-instantons, 
not just Hermitian--Einstein connections, and  $A_{Z}$ and
$A_{X_{n_j}}$ are $Spin(7)$-instantons.

\subsection{Approximate solution}

We identify a small ball around each $p_j \, \, (j= 1, \dots , k)$ in 
$Z$ with a small ball in $\R^8 / \Gamma_{8}$, and identify $E_Z$ with 
$E_{j}^{0}$ over the balls. Similarly, we identify the complement of a 
large ball around the origin of $X_{n_{j}}$ with the complement of a 
large ball around the origin of $(\R^8 / \Gamma_{8}) \setminus 0$, and 
identify $E_{X_{n_j}}$ with $E_{0}^{j}$ over those complements for each 
$j = 1, \dots , k$.  
We then glue $E_Z$ and $E_{X_{n_j}} \, ( j= 1, \dots , k)$ together by the 
above identifications, 
namely, 
$E_{Z} |_{M_{0}^{t}}$ is identified with $E_{X_{n_j}}|_{M_{j}^{t}}$ by 
$E_{Z} |_{M_{0}^{t} \cap M_{t}^{j}} \cong E_{0}^{j} |_{\text{annulus}} 
\cong E_{X_{n_{j}}} |_{M_{0}^{t} \cap M_{j}^{t}}$ on $M_{0}^{t} \cap 
M_{j}^{t}$ for each $j = 1, \dots , k$. 
We denote by $E$ the resulting vector bundle 
over $M$.

Next, we consider a smooth function $\chi : \R \to [0,1] $ with 
$\chi (x) = 0$ for $x \leq \frac{3}{4}$ and $\chi (x) = 1$ for $x \geq 
\frac{5}{6} $. 
We define $\chi_{t}^{j} (\rho)$ on $M_{0}^{t} \cap M_{j}^{t}$ by 
\begin{equation*}
 \chi_{t}^{j} (\rho) := \chi \left (\frac{\log \rho}{\log t} \right) , 
\end{equation*} 
where $\rho$ is the radius function defined in Section 2.3.  
Then we have 
\begin{equation*}
 \chi_{t}^{j} (\rho) = 
 \begin{cases}
 1 & \qquad ( \rho \leq t^{\frac{5}{6}}) ,\\ 
 0 & \qquad ( \rho \geq t^{\frac{3}{4}}) .
  \end{cases} 
\end{equation*}

We now define a connection $A_t$ on $E$ by 
$A_t = A_{Z}$ on $M_{0}^{t} \setminus \bigcup_{j=1}^{k} 
M_{j}^{t}$, $A_t = t^{*} (A_{X_{n_{j}}})$ on $M_{j}^{t} \setminus 
M_{0}^{t}$, and 
\begin{equation*} 
\begin{split}
 A_t &= \chi_{t}^{j} t^{*}(A_{X_{n_j}}) +  (1 - \chi_{t}^{j} ) A_{Z}  \\
 &= 
A_{0}^{j} + \chi_{t}^{j} t^{*} (A_{X_{n_j}} - A_{0}^{j}) 
+ (1 - \chi_{t}^{j} ) (A_{Z} - A_{0}^{j} ) 
\end{split}
\end{equation*} 
on $M_{0}^{t} \cap M_{j}^{t}$ for $j= 1, \dots ,k $.

\subsection{Estimate on the error}

In this section, we prove an estimate on the approximate solution 
(Proposition \ref{prop:error}).  
Since the $Spin(7)$-manifold $M=M^t$ depends on the parameter $t$ from  
the rescaling around the singular points, 
we use {\it scale-invariant norms} such as the 
$L^8$-norm for one-forms and $L^4$-norms for two-forms  
to obtain $t$-independent estimates.

\begin{proposition} 
Let $A$ be the approximate solution in Section 4.2. 
Then there exists a constant $C_1 >0$ independent 
 of $t$ such that 
\begin{equation*} 
 || \tilde{\pi}_{7}^2 (F_{A_t}) ||_{L^4} 
\leq  C_1 t^{\frac{1}{3}},  
\end{equation*} 
where $\tilde{\pi}^2_7$ is the projection with respect to the
 torsion-free 
$Spin(7)$-structure $\tilde{\Omega}^{t}$. 
\label{prop:error}
\end{proposition}

\begin{proof}
From \eqref{projomega}, we have 
\begin{equation*}
 | \tilde{\pi}^2_7 (F_{A_t}) |_{g^t} \leq 
| \pi^2_7 (F_{A_t}) |_{g^t}  
+ C | \tilde{\Omega}^{t} - \Omega^{t} |_{g^t} | F_{A_t} |_{g^t}, 
\end{equation*} 
where $\pi^2_7$ is the projection with respect to the
$Spin(7)$-structure $\Omega^{t}$ 
and $| \cdot |_{g^t}$ is a point-wise norm with respect to the 
metric $g^t$.  
Hence, raising to the fourth power and using the H\"{o}lder inequality, 
we obtain 
\begin{equation}
 || \tilde{\pi}^2_7 (F_{A_t}) ||_{L^4} \leq 
|| \pi^2_7 (F_{A_{t}}) ||_{L^4}  
+ C || \tilde{\Omega}^{t} - \Omega^{t} ||_{L^8} || F_{A_t} ||_{L^8},   
\label{eq:error}
\end{equation}
where $L^p$ norms are taken by the metric $g^t$.

We will prove  Proposition \ref{prop:error} 
by estimating each term in the right-hand-side of \eqref{eq:error}.

\begin{lemma}
\begin{equation}
 || \tilde{\Omega}^{t} -  \Omega^{t}  ||_{L^8} \leq C t^{\frac{4}{3}} .
\label{ineq:o8}
\end{equation} 
\end{lemma}

\begin{proof}
From Proposition 13.7.1 in Chapter 13 of \cite{MR1787733}, 
\begin{gather}
 || \tilde{\Omega}^{t} -  \Omega^{t} ||_{L^2} 
\leq  C t^{\frac{13}{3}} ,\\
  || \tilde{\Omega}^{t} -  \Omega^{t} ||_{C^0} \leq C  
 t^{\frac{1}{3}}.
 \label{ineq:o0} 
\end{gather}
Hence, for $p>2$ we have 
\begin{equation*}
  || \tilde{\Omega}^{t} -  \Omega^{t} ||_{L^p}
  \leq  
   || \tilde{\Omega}^{t} -  \Omega^{t} ||_{C^0}^{\frac{p-2}{p}} \, 
   || \tilde{\Omega}^{t} -  \Omega^{t} ||_{L^2}^{\frac{2}{p}}  
  \leq C t^{\frac{p-2}{3p}} 
  = C  t^{\frac{24 + p}{3p}} . 
\end{equation*}
In particular, $|| \tilde{\Omega}^{t} -  \Omega^{t} ||_{L^8} \leq C t^{\frac{4}{3}} $. 
\end{proof}

\begin{lemma}
Let $A_t$ be the approximate solution in Section 4.2. 
Then, 
\begin{equation}
 |\pi_{7}^2 (F_{A_t})| =  
\begin{cases}
O( t^6 \rho^{-8}) + O(1) , & \quad \rho \in (t^{\frac{5}{6}} , t^{\frac{3}{4}}), \\ 
0 , & \quad \text{otherwise} ,
\end{cases}
\label{eq:pi27fa}
\end{equation}
and 
\begin{equation}
 |F_{A_t}| =  
\begin{cases}
O(t^{6} \rho^{-8}), & \quad \rho \leq t^{\frac{5}{6}}, \\ 
O( t^6 \rho^{-8}) + O(1),  & \quad \rho \in (t^{\frac{5}{6}} , t^{\frac{3}{4}}), \\
O(1) , & \quad \rho \geq t^{\frac{3}{4}} ,
\end{cases}
\label{eq:fa}
\end{equation}
\end{lemma}

\begin{proof}
These follow from the definition of $A_{t}$, in particular, from $| d \chi_{t}^{j} | = O 
( r^{-1})$ and $| A_{X_{n_j}^{t}} - A_Z | = O (t^{6} r^{-7})$ on 
 $M_{0}^{t} \cap M_{j}^{t} \, (j = 1, \dots , k)$. 
\end{proof}

From \eqref{eq:pi27fa}, we obtain 
\begin{equation*}
\begin{split}
 || \pi_{7}^{2} (F_{A_t}) ||_{L^4} 
& \sim \left( \int_{t^{\frac{5}{6}}}^{t^{\frac{3}{4}}} 
(t^6 r^{-8} + 1 )^4 r^{7} dr \right)^{\frac{1}{4}} \\ 
& \sim \left( \left[ t^{24} r^{-24} + r^8 \right]_{t^{\frac{5}{6}}}^{t^{\frac{3}{4}}} 
\right)^{\frac{1}{4}} \\ 
& =  O (t) . \\ 
\end{split}
\end{equation*}
Also, from \eqref{eq:fa}, we obtain  
\begin{equation*}
 \begin{split}
 || F_{A} ||_{L^8} 
 &= \left[ \,  O (t^{-16}) O( t^{8} ) + \int_{t}^{1} (t^{6} r^{-8} 
 + 1 )^{8} r^7 dr  + O(1) \, \right]^{\frac{1}{8}} \\
 &= O(t^{-1}) . \\
 \end{split}
\end{equation*}
Hence  Proposition \ref{prop:error} follows. 
\end{proof}

\section{Linear problem}

In this section, we derive an estimate (Proposition 
\ref{prop:linearprob}) which comes from 
the Fredholm property of the linearized operator of the 
$Spin(7)$-instanton equation.

\subsection{Fredholm property of the linearized operator on ALE 
  $Spin(7)$-manifolds}

We use weighted Sobolev spaces on the ALE side  
in order to obtain the Fredholm property of the linearized operator 
from the direct use of 
the Lockhart-McOwen theory \cite{MR837256} (see also \cite{MR879560}, 
\cite{MR849427}).

\paragraph{Weighted Sobolev spaces.}

Let $X$ be an ALE $Spin(7)$-manifold. 
We denote by $\rho$ the radius function on $X$.   
Let $E \to X$ be a unitary vector bundle 
equipped with  a connection $A$ which is asymptotic to a flat connection 
at infinity. 
For $p \geq 1, k \geq 0$ and $\delta \in \R$, we define the 
{\it weighted Sobolev space} $L_{k, \delta}^{p} (E)$ 
by the set of locally integrable and $k$ 
times weakly differentiable section $f$ of $E$,  for which the norm  
\begin{equation*} 
 || f ||_{L^{p}_{k, \delta}} 
= \sum_{j=0}^{k} \left( 
\int_{X} \rho^{-8} | \rho^{-\delta +j} \nabla^{j}_{A} f |^{p}  d V 
\right)^{\frac{1}{p}}  
\end{equation*}
is finite. 
Then $L^{p}_{k, \delta} (E)$ is a Banach space. 
We remark the following relations between the scale-invariant norms 
mentioned above and the weighted norms:  
\begin{equation*} 
 || a ||_{L^{8}} = || a ||_{L^{8}_{-1}} , 
\, \,
 || \nabla_{A} a ||_{L^4} = || \nabla_{A} a ||_{L^{4}_{-2
}} , 
\, \,
|| F_{A} ||_{L^4} = || F_{A} ||_{L^{4}_{-2}}. 
\end{equation*}
We have the following Sobolev embedding theorem for the weighted spaces 
as well. 
\begin{proposition}[Sobolev embedding (\cite{MR879560}, Theorems 
 4.8)]
Let $k \geq l \geq 0, \, p, q  \geq 1$. 
If $\frac{1}{p} \leq \frac{1}{q} + \frac{k -l}{n}$, $\delta \leq \delta'$, 
then $L_{k, \delta}^{p} (E) \to L_{l , \delta'}^{q} (E)$ 
is a continuous
 inclusion. 
\end{proposition}

\paragraph{Fredholm property.}

We deduce the Fredholm property of the linearized operator on ALE 
spaces by using the Lockhart--McOwen theory.

Since we consider a connection asymptotic to a flat connection at 
infinity,  the linearized operator $L_{A}$  
reduces to the Dirac operator 
on $S^{7} /\Gamma_{8} \times (R, \infty )$ with the metric $r^2 g_{S^7} 
+ dr^2$ at infinity.  
Then the Lockhart--McOwen theory \cite{MR837256}, \cite{MR879560} 
tells us that the linearized operator 
$L_{A} : L^{p}_{k+1 , \delta} 
(\mathfrak{u}(E) \otimes \Lambda^1) \to L^{p}_{k , \delta -1}  
( \mathfrak{u} (E) \otimes ( \Lambda^{0} \oplus \Lambda^{2}_{7} ) )$ 
is Fredholm if and only if  $\delta$ does not lie in an exceptional set 
which is essentially determined by eigenvalues of the Dirac operator on 
$S^7 / \Gamma_{8}$ in our case. 
Using the fact that the eigenvalues of the Dirac operators on the sphere 
$S^n$ of constant sectional curvature $1$ are $\pm \left( \frac{n}{2} +
k \right) , \, k \geq 0$ (see for example Theorem 1 in
\cite{MR1361548}), 
we put 
\begin{equation*}
 \mathcal{D} 
= 
\left\{ 
\pm \left( \frac{7}{2}  + k \right) - \frac{7}{2} 
\, : \,  
k \geq 0 , \, k \in \Z 
\right\} .
\end{equation*}
Then the following is the direct consequence of Theorem 6.2 of 
\cite{MR837256}. 
\begin{proposition} 
Let $X$ be an ALE $Spin(7)$-manifold, 
$E$ a unitary vector bundle over $X$, and 
$A$ a $Spin(7)$-instanton on $E$ asymptotic to a flat connection $A_{0}$ 
 at infinity. 
Let $p \geq 1, k \geq 0$, and $\delta \in \R \setminus \mathcal{D}$.  
Then the operator 
\begin{equation*} 
 L_{A} : L_{k+1 , \delta}^{p} ( 
  \mathfrak{u}(E) \otimes \Lambda^{1} ) \to L_{k, \delta -1}^{p} 
( \mathfrak{u}(E) \otimes ( \Lambda^{0} \oplus \Lambda^{2}_{7} ) ) 
\end{equation*} 
is Fredholm. 
Moreover, the kernel, cokernel, and index are independent of 
$p,k$ and $\delta$ 
 within any connected component in $\R \setminus \mathcal{D}$. 
\label{prop:fredholm} 
\end{proposition}

\paragraph{Improvement of decay rates.}

By using the Fredholm property of the operator $L_{A}$ on ALE 
$Spin(7)$-manifolds, we prove the following:

\begin{proposition}

Let $X$ be an ALE $Spin(7)$-manifold, $E$ a unitary vector bundle over $X$, 
and $A_0$ a connection asymptotic at rate $\lambda$ to a flat connection at infinity. 
Assume that $a \in L^{8}_{k+1, \mu} (  \mathfrak{u}(E) \otimes 
\Lambda^1 )$, 
$\mu < -1$,  
and $A_{0} + a$  
satisfies the $Spin(7)$-instanton equation with $d_{A_0}^{*} a = 0$. 
Then $a \in L^{8}_{k+1 , \mu'}  ( \mathfrak{u}(E) \otimes \Lambda^1)$ 
for  any $\mu$ with $\lambda \leq \mu' < \mu$, which satisfies 
$[\mu', \mu ] \cap \mathcal{D} = \emptyset $. 
\end{proposition}

\begin{proof}
Firstly, the following is a consequence of Proposition 
 \ref{prop:fredholm}: 
\begin{lemma}
Let $a \in L_{k+1, \delta}^{8}  ( \mathfrak{u}(E) \otimes \Lambda^1)$, 
$\delta < -1$.  
Then for $\varepsilon \geq 2 \delta +1 $ with 
$\varepsilon \in \R \setminus \mathcal{D}$ 
there exists $\beta \in L_{k+1 , \varepsilon}^{8}  
( \mathfrak{u}(E) \otimes \Lambda^1)$ 
such that $L_{A_{0}} \beta = (0 , \pi_{7}^{2} (a \wedge a) )$ holds near 
 infinity. 
\label{lem:imp} 
\end{lemma}

\begin{proof}
Since $a \in L^{8}_{k+1, \delta}  (\mathfrak{u} (E) \otimes \Lambda^1)$,   
thus $\pi^{2}_{7} ( a  \wedge a ) \in L^{4}_{k +1, 2 \delta} 
(\mathfrak{u} (E) \otimes \Lambda^2) 
\subset L^{4}_{k +1, \varepsilon -1} 
(\mathfrak{u} (E) \otimes \Lambda^2)$.  
From Proposition \ref{prop:fredholm}, $L_{A_0} : L^{4}_{k+2 , \varepsilon} 
 (\mathfrak{u} (E) \otimes \Lambda^1) 
\to L^{4}_{k+1, \varepsilon -1 } 
   ( \mathfrak{u} (E) \otimes (\Lambda^0 \oplus  
\Lambda^2_{7}) )$ is Fredholm.  
Hence the cokernel of $L_{A_{0}}$ is 
 finite-dimensional, of dimension $n$ say. 
We take compactly supported sections $\varphi_1 , \dots , \varphi_{n}$  such that 
$$ L^{4}_{k+1 , \varepsilon -1} 
 \left(  \mathfrak{u} (E) \otimes (\Lambda^{0} \oplus \Lambda_{7}^{2} )  
\right) 
= L_{A_0} \left( L^{4}_{k+2 ,\varepsilon} 
(\mathfrak{u} (E) \otimes \Lambda^1) \right) 
\oplus \langle \varphi_1 , \dots ,  \varphi_n \rangle .$$ 
Then there exist $\beta \in L^{4}_{k+2 , \varepsilon } 
( \mathfrak{u} (E) \otimes \Lambda^1)$ and unique  
constants $u_1 , \dots , u_n$ such that  
$$ (0 ,  \pi^{2}_{7} (a \wedge a) ) 
= L_{A_0} \beta  
+ u_1 \varphi_1 + \cdots u_n \varphi_n . $$ 
Thus, 
we have $L_{A_0} \beta = (0 , \pi^{2}_{7} (a \wedge a) )$ 
outside the support of $\varphi_1 , \dots , \varphi_n$ in $X$.   
By Sobolev embedding $L_{k+2 , \varepsilon }^{4} (\mathfrak{u} (E) \otimes \Lambda^1) 
\to L_{k+1 , \varepsilon}^{8} (\mathfrak{u} (E) \otimes \Lambda^1)$,  
this $\beta$ lies in $L_{k+1 , \varepsilon}^{8} 
( \mathfrak{u} (E) \otimes \Lambda^1) $. 
\end{proof}

We take $\delta = \mu$ and 
$\varepsilon = \text{max} \{ 2 \mu +1 , \mu' \}$ in Lemma \ref{lem:imp}. 
Then we have $[ \varepsilon , \delta ] \cap \mathcal{D} = \emptyset$   
and $\varepsilon \in \R \setminus \mathcal{D}$,  
 since $[\mu' , \mu ] \cap \mathcal{D} = \emptyset$. 
Lemma \ref{lem:imp} gives $\beta \in L^{8}_{k +1 , \varepsilon} 
(\mathfrak{u}(E) \otimes \Lambda^{1})$ with $L_{A_{0}} 
\beta = ( 0 , \pi^{2}_{7} (a \wedge a ))$ near infinity. 
Thus, we get $L_{A_{0}} (a + \beta) =0$ near infinity. 
Since the cokernel of $L_{A_{0}}$ 
is independent of the choice of weights within 
any component in $\R \setminus \mathcal{D}$,
$L_{A_{0}} (a + \beta) \perp (\text{coker} \, L_{A_{0}})_{\delta}$ 
implies 
$L_{A_{0}} (a + \beta) \perp (\text{coker} \, L_{A_{0}})_{\varepsilon}$.  
Thus, $L_{A_0} (a + \beta ) \in (\text{Im} L_{A_{0}})_{\varepsilon}$, 
and there exists $\gamma \in 
L_{k+1 , \varepsilon}^{8} (\mathfrak{u}(E) \otimes \Lambda^{1})$ such that 
$L_{A_0} \gamma = L_{A_0} (a + \beta )$. 
As the kernel of $L_{A_{0}}$  
 is also independent of the choice of weights within any component in $\R \setminus \mathcal{D}$, 
we obtain, 
$a + \beta - \gamma \in ( \ker L_{A_0} )_{\delta} 
=   ( \ker L_{A_0} )_{\varepsilon}$. 
Since $\beta , \gamma \in L^{8}_{k +1 , \varepsilon} 
(\mathfrak{u}(E) \otimes \Lambda^{1})$, 
thus $a \in L^{8}_{k +1 , \varepsilon} 
(\mathfrak{u}(E) \otimes \Lambda^{1})$.

Therefore, starting with $a \in L_{k+1, \mu}^{8} 
(\mathfrak{u}(E) \otimes \Lambda^{1})$ with $\mu < -1$, we 
 see that  $a \in L_{k+1 , 2\mu +1}^{8}
(\mathfrak{u}(E) \otimes \Lambda^{1})$, provided $[ 2 \mu +1 , \mu ] 
\cap \mathcal{D} = \emptyset$. 
We then put $\mu_{0} = \mu = (\mu + 1) -1 , \, \mu_1 = 2\mu +1 = 2 (\mu +1) -1 , \, \dots , 
\, \mu_{k} = 2^{k} (\mu +1 ) -1 $, and let $\ell$ be the least 
 satisfying 
 $2^{\ell} (\mu+1) -1 \leq \mu'$, and say $\mu_{\ell} = \mu'$ for 
simplicity. 
Since $[\mu' , \mu ] \cap \mathcal{D} = \emptyset$, thus  
$\mu_{0} ,  \dots , \mu_{\ell} \in \R \setminus \mathcal{D}$.  
Hence, we inductively obtain 
$a \in L_{k+1 , \mu_{i}}^{8} 
(\mathfrak{u}(E) \otimes \Lambda^{1}) 
\, (i =1 , \dots , \ell)$. 
Thus, $a \in L_{k+1 , \mu'}^{8} (\mathfrak{u}(E) \otimes \Lambda^{1})$. 
\end{proof}

\subsection{Estimates}

We choose a finite dimensional vector space $K_Z$ in $ 
\Omega^1 (Z,  \mathfrak{u} (E_Z) )$, whose elements are 
supported away from $p_j  \, (j=1 , \dots , k)$ with the following 
properties: 
\begin{itemize}
 \item $ \dim K_Z = \dim \ker (L_{A_Z})$  and 
 \item $ 
\Omega^1 (Z , \mathfrak{u}(E_Z) )  = 
K_{Z}^{\perp}  \oplus \ker (L_{A_Z})$, 
\end{itemize}
where $K_{Z}^{\perp}$ is the $L^2$-orthogonal complement of $K_Z$ in 
$\Omega^1 (Z , \mathfrak{u}(E_Z) )$. 
Since all elements in $K_Z$ are supported on the region $M_{0}^{t}$ of
$Z \subset 
M$ for small $t$, we can think of $K_Z$ as lying in $ \Omega^1 (M , \mathfrak{u}(E) )$. 
We will use $K_Z$ as a substitute for the kernel of $L_{A_{Z}}$, which
also makes sense on $M$.

In a similar way to 
$K_Z$ above,  
we choose $K_{X_{n_j}}$ in $\Omega^1 (X_{n_{j}}, 
\mathfrak{u} (E_{X_{n_j}}) )$ for each $j = 1 , \dots , k$, 
whose elements are compactly supported away from infinity, 
and think of $K_{X_{n_j}}$ in $ 
\Omega^1 (M , \mathfrak{u}(E) )$.  
These $K_Z, K_{X_{n_j}} (j = 1, \dots , k)$ are substitutes for the 
kernels of $L_{Z} , L_{X_{n_j}} (j = 1 , \dots , k)$. 
We then put  
\begin{equation*} 
K = K_{Z} \oplus \bigoplus_{j=1}^{k} K_{X_{n_j}} 
\subset \Omega^{1} (M , \mathfrak{u} (E)). 
\end{equation*}

Firstly, we prove the following: 
\begin{lemma}
There exists a constant 
$C_2 >0$ such that 
the following holds  
for any $a_{X_{n_j}} \in L^{4}_{1, -1} 
(\mathfrak{u} (E_{X_{n_j}}) \otimes \Lambda^1 )$,  
which is $L^2$-orthogonal to $K_{X_{n_{j}}}$:   
\begin{equation}
  || a_{X_{n_j}}||_{L^{8}} + || \nabla a_{X_{n_j}} ||_{L^{4}} 
\leq C_2 || L_{A_{X_{n_{j}}}} a_{X_{n_j}} ||_{L^{4}}. 
\label{eq:wtf}
\end{equation}
\label{lem:fxnj}
\end{lemma}

\begin{proof}
From Proposition \ref{prop:fredholm}
\begin{equation*}
 L_{A_{X_{n_{j}}}} 
 : L_{1 , -1}^{4} ( 
 \mathfrak{u} (E_{X_{n_{j}}}) \otimes \Lambda^1) \to L_{-2}^{4}  
( \mathfrak{u} (E_{X_{n_{j}}})  \otimes 
( \Lambda^{0} \oplus \Lambda^{2}_{7} )  ) 
\end{equation*} 
is Fredholm, thus, if $a_{X_{n_j}} \perp K_{X_{n_j}}$, we have 
\begin{equation*}
 || a_{X_{n_j}} ||_{L^{4}_{1, -1}} \leq C || L_{A_{X_{n_{j}}}} a_{X_{n_j}}
 ||_{L^{4}_{-2}} . 
\end{equation*}
From the Sobolev embedding $L^{4}_{1, -1} \to L^{8}_{-1}$, we obtain 
\begin{equation*}
 || a_{X_{n_j}} ||_{L^{8}_{-1}} \leq C \left( ||a_{X_{n_j}} ||_{L^{4}_{-1}} 
+ || \nabla a_{X_{n_j}} ||_{L^{4}_{-2}} 
\right) . 
\end{equation*}
Thus,  we get 
\begin{equation*}
 || a_{X_{n_j}} ||_{L^{8}_{-1}} + || \nabla a_{X_{n_j}} ||_{L^{4}_{-2}} 
\leq C || L_{A_{X_{n_{j}}}} a_{X_{n_j}} ||_{L^{4}_{-2}} . 
\end{equation*}
Since $|| \cdot ||_{L^{8}_{-1}} = || \cdot ||_{L^8} , \, 
|| \cdot ||_{L^{4}_{-2}} = || \cdot ||_{L^4}$, 
hence, \eqref{eq:wtf} follows. 
\end{proof}

Similarly we have the following: 
\begin{lemma} 
There exists a constant 
$C_3 >0$ such that the following holds for any  
$a_Z \in L^{4}_{1} (\mathfrak{u} (E_{Z}) \otimes \Lambda^1)$, which is 
 $L^2$-orthogonal to $K_Z$: 
\begin{equation}
  || a_Z ||_{L^{8}} + || \nabla a_Z ||_{L^{4}} 
\leq C_3 || L_{A_{Z}} a_Z ||_{L^{4}}  . 
\label{wtf2}
\end{equation}
\label{lem:fz}
\end{lemma}

With these lemmas above in mind, we prove the following:

\begin{proposition}
There exists a constant 
$C_4 >0$ independent of $t$ such that 
if $t$ is sufficiently small and $a \in L^{4}_{1} (\mathfrak{u}(E) \otimes \Lambda^1)$ 
is $L^2$-orthogonal to $K$, then   
\begin{equation*}
|| a ||_{L^{8}} + || \nabla a ||_{L^4}   
\leq C_4 \, || L_{A_t} a ||_{L^{4}}, 
\end{equation*}
where $L_{A_t}$ is the linearized operator with respect to the 
 $Spin(7)$-structure 
$\Omega^{t}$ on $M$. 
\label{prop:plinearprob} 
\end{proposition}

\begin{proof}
We decompose $a \in \Omega^{1} (M , \mathfrak{u} (E))$ as 
$$ a = \sum_{j=1}^{k} \chi^{j}_{t} a + \left( 1 - \sum_{j=1}^{k} \chi^j_{t} \right) a , $$ 
where $\chi_{t}^{j}$ is the cut-off function around each $p_{j} \,(j=1 ,\dots ,k)$, 
defined in 
 Section 4.2.

Since we use the conformally-invariant norms, the same inequalities 
as \eqref{eq:wtf} and \eqref{wtf2} hold 
 on the regions in $M^{t}$, which are isomorphic to $X_{n_j} \, (j=1 , \dots , k)$ 
and $Z$, namely, 
we have  
\begin{equation*}
\begin{split}
 || \chi^j_{t} a ||_{L^{8}} 
+  || \nabla ( \chi^j_{t}  a ) ||_{L^{4}} 
\leq C_2  \, || L_{A_{X_{n_{j}}}} (\chi^{j}_{t}  a) ||_{L^{4}} 
\end{split}
\end{equation*}
for each $j = 1, 2, \dots , k$, and 
\begin{equation*}
\begin{split}
 || (1 -\sum_{j=1}^{k} \chi^{j}_{t})  a ||_{L^{8}} 
+ || \nabla (1 - \sum_{j=1}^{k} \chi^{j}_{t})   a ||_{L^{4}} 
\leq C_3 \, || L_{A_{Z}} (1 - \sum_{j=1}^{k} \chi^{j}_{t}) a  ||_{L^{4}}. 
\end{split}
\end{equation*}
Therefore, 
\begin{equation}
 \begin{split}
 || a ||_{L^{8}} 
&+  || \nabla a ||_{L^{4}} \\
&\leq C_2  \sum_{j=1}^{k}    
 || L_{A_{X_{n_j}}} ( \chi^{j}_{t} a ) ||_{L^{4}}  
  +  C_3 || L_{A_{Z}} ( 1- \sum_{j=1}^{k} \chi^{j}_{t}) a  ||_{L^{4}} \\ 
 &\leq C_2 
 \sum_{j=1}^{k} || \chi^{j}_{t} \left( L_{A_{X_{n_{j}}}} a \right) ||_{L^4} 
   + C \sum_{j=1}^{k} || d \chi^{j}_{t}  \wedge a ||_{L^4} \\
 & \qquad + C_3 \, || ( 1 - \sum_{j=1}^{k} \chi^{j}_{t} ) 
L_{A_Z} a ||_{L^4} +  C \sum_{j=1}^{k} || d \chi^{j}_{t} \wedge a ||_{L^4} . 
  \end{split}
\label{ineq:plinearprob}
\end{equation}
In order to prove Proposition \ref{prop:plinearprob}, we estimate each
 term of the final two lines of \eqref{ineq:plinearprob}.

From the H\"{o}lder inequality 
\begin{equation*}
\begin{split}
 ||  d \chi^j_{t} \wedge a ||_{L^4} 
 &\leq || d \chi^j_{t} ||_{L^8}   || a ||_{L^8} .  \\ 
\end{split}
\end{equation*}
Since $| d \chi^{j}_{t} | \sim \frac{1}{r | \log t |}$ on 
$r \in ( t^{\frac{5}{6}} , t^{\frac{3}{4}} )$, 
\begin{equation*}
  || d \chi^{j}_{t} ||_{L^8} 
\leq C 
\left( \int_{t^{\frac{5}{6}}}^{t^{\frac{3}{4}}} 
\left( \frac{1}{r | \log t | } \right)^8 
r^7 dr \right)^{\frac{1}{8}} 
\leq C  \left( \frac{1}{|\log \, t |} \right)^{\frac{7}{8}} . 
\end{equation*}
Hence $|| d \chi^j_{t} \wedge a ||_{L^4}$ has the order of 
$ O \left( |\log t |^{- \frac{7}{8}} 
\right)$.

On the other hand, we have 
\begin{equation*}
 L_{A_t} a 
=  L_{A_{X_{n_{j}}}} a 
- (1 - \chi^{j}_{t}) \pi^{2}_{7} \left( (A_{Z} - A_{X_{n_j}}) \wedge a \right) . 
\end{equation*}
Thus, 
\begin{equation*}
\chi^{j}_{t}  \left( L_{A_t} a \right) 
= \chi^{j}_{t} \left( L_{A_{X_{n_{j}}}} a \right) 
- \chi^{j}_{t} (1 - \chi^{j}_{t}) 
\pi^{2}_{7} \left( (A_{Z} - A_{X_{n_j}}) \wedge a \right) . 
\end{equation*}
Therefore, 
\begin{equation*}
 \begin{split}
 ||  \chi^{j}_{t} ( L_{A_{X_{n_j}}} a ) ||_{L^4}
&\leq ||  \chi^{j}_{t} ( L_{A_t} a )||_{L^4} + || 
\chi^{j}_{t} ( 1- \chi^{j}_{t}) 
\pi^{2}_{7} (
(A_{Z} - A_{X_{n_{j}}}) \wedge a) ||_{L^4} . 
 \end{split}
\end{equation*}
Here,  $| \chi^{j}_{t} ( 1- \chi^{j}_{t}) 
(A_{Z} - A_{X_{n_{j}}}) | $ has the order of $O (t^{6} r^{-7})$, thus, 
$|| \chi^{j}_{t} ( 1- \chi^{j}_{t}) \pi^{2}_{7} (
(A_{Z} - A_{X_{n_{j}}}) \wedge a )||_{L^4} 
= O (t) || a ||_{L^8}$.

Similarly, 
\begin{equation}
 \begin{split}
|| &(1 - \sum_{j=1}^{k} \chi^{j}_{t}) L_{A_{Z}} a ||_{L^4} \\
 &\leq || (1 - \sum_{j=1}^{k} \chi^{j}_{t} ) L_{A_t} a ||_{L^4} 
 + ||  (1 - \sum_{j=1}^{k} \chi^{j}_{t} ) 
 \sum_{j=1}^{k} (\chi^j_{t}  \pi^{2}_{7} 
( (A_{Z} - A_{X_{n_j}} ) \wedge a) ) ||_{L^4} . \\
 \end{split}
\label{eq:chiaz}
\end{equation}
Again, the last term of the right-hand side of  \eqref{eq:chiaz} has the order of  $O (t) \, ||a ||_{L^8}$. 
Hence Proposition \ref{prop:plinearprob} follows. 
\end{proof}

We now prove the following: 
\begin{proposition}
There exists a constant 
$C_5 >0$ independent of $t$ such that 
if $a \in L^{4}_{1} (\mathfrak{u} (E) \otimes \Lambda^1)$ 
is $L^2$-orthogonal to $K$, then 
the following holds for $t$ sufficiently small: 
\begin{equation*}
|| a ||_{L^{8}} + || \nabla a ||_{L^4}   
\leq C_5 \, || \tilde{L}_{A_{t}} a ||_{L^{4}},   
\end{equation*}
where $\tilde{L}_{A_{t}}$ is the linearized operator with respect to the torsion-free
 $Spin(7)$-structure 
$\tilde{\Omega}^{t}$ on $M$. 
\label{prop:linearprob}
\end{proposition}

\begin{proof}
$\tilde{L}_{A_{t}} a $ may be written by 
\begin{equation*}
 \tilde{L}_{A_{t}} a 
 = L_{A_t} a   + S \cdot a + T \cdot  \nabla a ,
\end{equation*}
where $S$ and $T$ are tensor fields with 
\begin{gather*}
|S| \sim |\tilde{\Omega}^{t} -  \Omega^{t} | 
+ | \nabla (\tilde{\Omega}^{t} -  \Omega^{t} ) | , \\ 
|T| \sim |\tilde{\Omega}^{t} - \Omega^{t} |. 
\end{gather*}
Hence, 
\begin{equation*}
 \begin{split}
 C_4 || \tilde{L}_{A_{t}} a  ||_{L^4}
 &= C_4 || \tilde{L}_{A_{t}} a + S \cdot a 
+ T \cdot \nabla a ||_{L^4} \\ 
 &\geq C_4 || \tilde{L}_{A_{t}} a  ||_{L^4} 
 - C_4 || S \cdot  a ||_{L^4} 
 - C_4 || T \cdot  \nabla a ||_{L^4} \\ 
 & \geq  ( ||a ||_{L^8} + || \nabla a ||_{L^4} )  
 - C_4 || S ||_{L^8} ||a ||_{L^8} 
 - C_4 || T ||_{C^0} || \nabla a ||_{L^4} \\ 
 &= (1 - C_4  || S ||_{L^8} ) ||a ||_{L^8} 
   + (1 - C_4 || T ||_{C^{0}} ) || \nabla a ||_{L^4} .\\   
 \end{split}
\end{equation*}

From Theorem \ref{th:spinss}, we have  
\begin{equation}
 || \nabla (\tilde{\Omega}^{t} -  \Omega^{t}) ||_{L^8} \leq 
C t^{\frac{2}{15}}. 
\label{eq:edomega}
\end{equation}
Thus, from \eqref{ineq:o8}, \eqref{ineq:o0} and \eqref{eq:edomega}, we have 
\begin{equation*}
 ||S ||_{L^8} \leq C t^{\frac{2}{15}} , \quad || T ||_{C^0} \leq 
C t^{\frac{1}{3}} . 
\end{equation*}
Therefore, if we take $t$ small enough so that  
$ C_4 || S ||_{L^8} \leq \frac{1}{2} , 
\, C_4 || T ||_{C^0} \leq \frac{1}{2} $ hold, 
then we 
 obtain  
\begin{equation*}
 || a ||_{L^8} + || \nabla a ||_{L^4} 
\leq 2 C_4 || \tilde{L}_{A_{t}} a  ||_{L^4} .  
\end{equation*}
This completes the proof of Proposition \ref{prop:linearprob}. 
\end{proof}

The following is the direct corollary of Proposition
\ref{prop:linearprob}:

\begin{Corollary}
$ \ker \tilde{L}_{A_t} \cap K^{\perp} = \{ 0 \} $. 
Hence, 
$ \dim \ker \tilde{L}_{A_t} \leq \dim K 
= \dim \ker L_{Z} + \sum_{j=1}^{k} \dim \ker L_{X_{n_j}} $. 
\label{cor:ker}
\end{Corollary}

We now assume that the linearized operators $L_{A_{Z}} ,\, 
L_{A_{X_{n_{j}}}} \, ( j= 1, \dots ,k)$ satisfy the condition in 
Section 4.1, 
namely, the cohomology $H^{2} (Z , \mathfrak{u} (E))$ of the complex 
\eqref{eq:complex} vanishes, but $H^{0} (Z , \mathfrak{u}(E))$ of the
complex 
\eqref{eq:complex} does not necessarily vanish,  
and $ L_{A_{X_{n_j}}} : L_{1,
\delta}^{4} (\mathfrak{u} (E_{X_{n_{j}}}) \otimes \Lambda^{1} (X_{n_{j}}) ) 
\to L_{\delta -1}^{4}  
(\mathfrak{u} (E_{X_{n_{j}}}) \otimes ( \Lambda^{0} (X_{n_{j}}) \oplus \Lambda_{7}^{2} 
(X_{n_{j}}) ))
\, (j= 1, 2, \dots , k)$ 
for $\delta \in ( -7, 0)$ 
is surjective.

We then consider a finite dimensional vector space $C_Z$ in $ 
\Omega^0 (Z,  \mathfrak{u} (E_Z) )$, whose elements are 
supported away from $p_j  \, (j=1 , \dots , k)$ with the following 
properties: 
\begin{itemize} 
 \item $ \dim C_Z = \dim \ker (L_{A_Z}^{*})$  and 
 \item $ 
\Omega^0 (Z , \mathfrak{u}(E_Z) )  = 
C_{Z}^{\perp}  \oplus \ker (L_{A_Z}^{*})$, 
\end{itemize} 
where $C_{Z}^{\perp}$ is the $L^2$-orthogonal complement of $C_Z$ in 
$\Omega^0 (Z , \mathfrak{u}(E_Z) )$. 
Since all elements in $C_Z$ are supported on the region $M_{0}^{t}$ of
$Z \subset M$ for small $t$, we can think of $C_Z$ as lying in $ \Omega^0
 (M , \mathfrak{u}(E) )$. 
We choose this $C_Z$ in the following way.  
Firstly, by using the method of Proposition \ref{prop:linearprob}, 
one can show that there exists a constant $C_6 >0$ such
that if $a \in L^{4}_{1} (\mathfrak{u}(E) \otimes \Lambda^1)$ with $a
\perp K$, then the following holds for $t$ sufficiently small: 
\begin{equation}
 ||a||_{L^2} \leq C_6
|| \tilde{L}_{A_t} a ||_{L^2} . 
\label{eq:cz1}
\end{equation}
We then choose $C_Z$ such that the following holds for all $c \in C_Z$: 
\begin{equation*}
 || L_{A_Z}^{*} c ||_{L^2}\leq \frac{1}{4 C_6} 
|| c ||_{L^2} . 
\end{equation*} 
This holds, provided $C_Z$ is sufficiently close to $\ker L_{A_Z}^{*}$
in $L^{2}_{1}$. 
Note that, by taking $t$ sufficiently small, we obtain 
\begin{equation}
 || \tilde{L}_{A_t}^{*} c ||_{L^2}\leq \frac{1}{2 C_6} 
|| c ||_{L^2} 
\label{eq:cz2}
\end{equation}
for all $c \in C_{Z}$. 
We will use $C_Z$ as a substitute for the kernel of $L_{A_Z}^{*}$, which
also makes sense on $M$. 

Now, as for Lemmas \ref{lem:fxnj} and \ref{lem:fz}, we have:  
\begin{lemma}
There exists a constant 
$C >0 $ such that the following holds for any 
$( a_{X_{n_j}}, b_{X_{n_j}}) \in L^{4}_{1, -1} 
\left( \mathfrak{u}(E_{X_{n_j}}) \otimes 
( \Lambda^0 (X_{n_j}) \oplus \Lambda^{2}_{7} (X_{n_j}) )
\right)$: 
\begin{equation*}
  || a_{X_{n_j}}||_{L^{8}} + || \nabla a_{X_{n_j}} ||_{L^{4}} 
+ || b_{X_{n_j}}||_{L^{8}} + || \nabla b_{X_{n_j}} ||_{L^{4}} 
\leq C || L_{A_{X_{n_{j}}}}^{*} ( a_{X_{n_j}}, b_{X_{n_j}}) ||_{L^{4}}. 
\end{equation*}
\end{lemma}
\begin{lemma} 
There exists a constant 
$C > 0$ such that the following holds for any 
$(a_Z, b_Z) \in L^{4}_{1} 
\left( \mathfrak{u}(E_{Z}) \otimes ( \Lambda^0 (Z) \oplus \Lambda^{2}_{7} (Z) ) 
\right)$  
with $a_Z \perp C_{Z}$: 
\begin{equation*}
  || a_Z ||_{L^{8}} + || \nabla a_Z ||_{L^4} 
+ || b_Z ||_{L^{8}} + || \nabla b_Z  ||_{L^4} 
\leq C  || L_{A_{Z}}^{*} (a_Z , b_Z )||_{L^{4}}. 
\end{equation*}
\end{lemma}

Since the argument for 
Proposition \ref{prop:linearprob} also works 
for the formal adjoints $ L_{A_{Z}}^{*}, 
L_{A_{X_{n_j}}}^{*} (j= 1, \dots , k)$, and 
$\tilde{L}_{A_{t}}^{*}$,  therefore we obtain the 
following: 
\begin{proposition}
There exists a constant 
$C_7 > 0$ 
independent of $t$ such that 
if  $(a, b) \in L^{4}_{1}
\left( 
\mathfrak{u}(E) \otimes ( \Lambda^0 (M) \oplus \Lambda^{2}_{7} (M) ) 
\right)$
with $a \perp C_{Z}$, then 
\begin{equation*} 
|| a||_{L^{8}} + || \nabla a ||_{L^{4}} 
+ || b ||_{L^{8}} + || \nabla b  ||_{L^{4}} 
\leq C_7 || \tilde{L}_{A_{t}}^{*} (a , b )||_{L^{4}}. 
\end{equation*}
\label{prop:wtfd} 
\end{proposition}

We also have the following: 
\begin{proposition}
\begin{equation}
 L^4 \left( \mathfrak{u} (E) \otimes 
(\Lambda^0 \oplus \Lambda^{2}_{7} ) \right) 
= C_Z \oplus \tilde{L}_{A_t} 
\left( K^{\perp} \cap L^{4}_{1} 
(\mathfrak{u} (E) \otimes \Lambda^1 ) \right).
\label{eq:coker}
\end{equation}
\end{proposition}

\begin{proof}
Firstly, we prove $C_{Z} \cap \tilde{L}_{A_{t}} \left( K^{\perp} \cap L^{4}_{1} 
(\mathfrak{u} (E) \otimes \Lambda^1 ) \right) = \{ 0 \}$.  
Suppose for a contradiction that there exists 
$a \in  K^{\perp} \cap L^{4}_{1} 
(\mathfrak{u} (E) \otimes \Lambda^1 )$ such that $\tilde{L}_{A_t} a = c 
$ for some $c \in C_{Z}$ with $c \neq 0$. Then we have 
\begin{equation*}
 || c ||_{L^2}^{2}
= \langle c , c  \rangle 
= \langle c , \tilde{L}_{A_t} a \rangle 
= \langle \tilde{L}_{A_t}^{*} c , a \rangle 
\leq || \tilde{L}_{A_t}^{*} c ||_{L^2} || a ||_{L^2} . 
\end{equation*}
Thus, from \eqref{eq:cz1} and \eqref{eq:cz2}, we get 
$$ || c ||_{L^2}^{2} \leq \frac{1}{2} || c ||_{L^2}^{2} .$$
This is a contradiction. Hence  
$C_{Z} \cap \tilde{L}_{A_{t}} \left( K^{\perp} \cap L^{4}_{1} 
(\mathfrak{u} (E) \otimes \Lambda^1 ) \right) = \{ 0 \}$.

Next, by using the index theory, one can obtain that 
$$ \text{Ind} ( \tilde{L}_{A_t} ) 
= \text{Ind} \, (L_{Z}) + \sum_{j=1}^{k} 
\text{Ind} \, (L_{X_{n_j}}).
$$ 
Hence, 
\begin{equation}
 \text{Ind} ( \tilde{L}_{A_t} ) 
=  \left( \dim K_Z - \dim C_{Z} \right) 
     + \sum_{j=1}^{k} \dim K_{X_{n_j}} 
= \dim K - \dim C_{Z} . 
\label{eq:indexc}
\end{equation}
On the other hand, $\tilde{L}_{A_t} (K^{\perp}) \subset \text{Im} \, \tilde{L}_{A_t}$ 
has the codimension 
$$ \dim K - \dim \ker \tilde{L}_{A_t}, $$ 
and $\text{Im} \, \tilde{L}_{A_t} \subset 
L^4 \left( \mathfrak{u} (E) \otimes 
(\Lambda^0 \oplus \Lambda^{2}_{7} ) \right) $
 has the codimension $\dim \ker \tilde{L}_{A_t}^{*}$. 
Thus, the codimension of $\tilde{L}_{A_t} (K^{\perp}) \subset 
L^4 \left( \mathfrak{u} (E) \otimes 
(\Lambda^0 \oplus \Lambda^{2}_{7} ) \right) $ is 
$$ ( \dim K - \dim \ker \tilde{L}_{A_t} ) 
+ \dim \ker \tilde{L}_{A_t}^{*}. $$ 
This is $ \dim C_{Z}$ by \eqref{eq:indexc}. 
Hence,  \eqref{eq:coker} holds. 
\end{proof}

\section{Construction}

In this section, we prove the following: 
\begin{theorem} 
Let $M=M^t$ be the torsion-free $Spin(7)$-manifold in Section 2, that 
is, $M$ is 
a desingularization of a Calabi--Yau four-orbifold $Y$ with finitely 
many singular points and an anti-holomorphic involution fixing only the 
singular set by gluing ALE $Spin(7)$-manifold 
$X_{n_j} \, (j= 1,2, \dots , k)$ at 
each singular points $p_{j} \, (j = 1,2, \dots , k)$. 
Assume that there are Hermitian--Einstein connections on 
$Y$ and $W_{n_j}$'s satisfying the conditions in Section 4.1. 
Then there exists a $Spin(7)$-instanton on 
a vector bundle $E$ over $M =M^t$ for $t$ sufficiently small. 
\label{main} 
\end{theorem}

In Section 6.1, we find a $Spin(7)$-instanton $A_t + a_{t}$ in $L_{1}^{4}$ 
by an iterative method, using the estimates in Section 4 and 5.  
The regularity of the solution is given in Section 6.2.

\subsection{Inductive construction}

The equation we would like to solve is 
\begin{equation}
 \tilde{L}_{A_t} 
a_{t} = \left( 
c_t ,- \tilde{\pi}_{7}^{2} (F_{A_t}) - \tilde{\pi}_{7}^{2} (a_t \wedge a_t)
\right) 
\label{eq:ceq}
\end{equation} 
with $c_t \in  C_{Z}$. 
From \eqref{eq:coker}, for a given $e \in L^{4} \left(  
\mathfrak{u} (E) \otimes \Lambda^{2}_{7} \right)$, 
there exists a unique $c \in C_Z$ such that $(c, e) \in
\tilde{L}_{A_t} (K^{\perp})$. 
We then inductively define a sequence $\{ a_{t}^{k} \}$ and 
$\{ c_{t}^{k}\} \, (k= 0, 1, 2, 
\dots)$ by 
\begin{equation} 
 \tilde{L}_{A_t} 
 a_{t}^{k+1} = \left( c_{t}^{k} ,   
- \tilde{\pi}_{7}^{2} (F_{A_t}) - \tilde{\pi}_{7}^{2} (a_{t}^{k} 
\wedge a_{t}^{k})  \right), 
\label{eq:series} 
\end{equation} 
with $a_{t}^{0} = 0$, and each $a_{t}^{k} \, ( k= 0,1 ,2 , \dots)$ 
is uniquely determined 
by the condition $a^{k}_t \perp K$,  
and each $c_{t}^{k}$ is uniquely determined by 
the right-hand-side of \eqref{eq:series} lying in 
$\tilde{L}_{A_t} (K^{\perp})$.

\begin{lemma} 
Assume that  $e \in L^{4} \left(  
\mathfrak{u} (E) \otimes \Lambda^{2}_{7} \right)$ and 
$c \in C_{Z}$ satisfy $(c,e) \in \tilde{L}_{A_t} (K^{\perp})$. 
Then the following holds: 
\begin{equation}
 || c ||_{L^{4}_{1}} \leq C_8  ||e ||_{L^4}.
\label{eq:estc} 
\end{equation}
\label{estcoker}
\end{lemma}

\begin{proof}
Each $c \in C_{Z}$ uniquely extends to 
$( c + a' , b') \in \left( \tilde{L}_{A_t} (K^{\perp}) \right)^{\perp}$,  
where $a' \in C_{Z}^{\perp} \cap L^{4}_1 \left(  
\mathfrak{u} (E) \otimes \Lambda^{0} \right) $ and 
$b' \in L^{4}_{1} \left(  
\mathfrak{u} (E) \otimes \Lambda^{2}_{7} \right)$. 
Since $(c ,e) \in \tilde{L}_{A_t} (K^{\perp})$, we obtain 
$\langle (c,e) , (c +a' , b') \rangle_{L^2} =0$. 
Thus we get  
\begin{equation}
 \begin{split}
 || c ||_{L^2}^{2} 
 &= - \langle b' , e \rangle_{L^2} \\
 &\leq || b' ||_{L^2} || e ||_{L^2} .\\ 
 \end{split}
\label{eq:ip}
\end{equation}
Since $a' \perp C_{Z}$, from Proposition \ref{prop:wtfd} we get  
\begin{equation*}
 \begin{split}
   || b' ||_{L^8} 
&\leq C_7 || \tilde{L}_{A_{t}}^{*} (a' , b') ||_{L^4} \\ 
&= C_7 || \tilde{L}_{A_{t}}^{*} (-c , 0 ) ||_{L^4} \\ 
&\leq C || c ||_{L^{4}_{1}} .  \\
 \end{split}
\end{equation*}
As $C_Z$ is a finite dimensional vector space, 
$|| c ||_{L^{4}_{1}} \leq C || c ||_{L^2}$ for all $c \in C_{Z}$. 
Hence 
\begin{equation}
  || b' ||_{L^8} 
\leq C || c ||_{L^2}. 
\label{eq:ip2}
\end{equation} 
Therefore, from \eqref{eq:ip} and \eqref{eq:ip2}, we obtain 
\begin{equation*}
 || c ||_{L^2}^{2} \leq C || c ||_{L^{2}} || e ||_{L^2} . 
\end{equation*}
Thus, $||c ||_{L^2} \leq C || e ||_{L^2}$. 
Again, using $|| c ||_{L^{4}_{1}} \leq C || c ||_{L^2}$ and 
$|| e ||_{L^2} \leq C || e ||_{L^4}$, we obtain \eqref{eq:estc}. 
\end{proof}

Proposition \ref{prop:linearprob} and Lemma \ref{estcoker} 
together with equation \eqref{eq:series} give us  
\begin{equation}
 \begin{split}
&|| a_{t}^{k+1} - a_{t}^{k} ||_{L^8} 
 + || \nabla ( a_{t}^{k+1} - a_{t}^{k} ) ||_{L^4} \\
& \qquad \qquad \leq 
C_5  || c_{t}^{k} - c_{t}^{k-1} ||_{L^4} 
+ C_5  || \tilde{\pi}_{7}^{2} (a_{t}^{k} \wedge a_{t}^{k} 
- a_{t}^{k-1} \wedge a_{t}^{k-1} ) ||_{L^4} \\ 
& \qquad \qquad 
\leq C_5 (C_{8} + 1 ) 
|| \tilde{\pi}_{7}^{2} (a_{t}^{k} \wedge a_{t}^{k} 
- a_{t}^{k-1} \wedge a_{t}^{k-1} ) ||_{L^4} \\ 
& \qquad \qquad 
\leq C_5 (C_{8} + 1 ) 
 || (| a_{t}^{k} - a_{t}^{k-1} |) ( |a_{t}^{k}| + |a_{t}^{k-1} |) ||_{L^4} \\
& \qquad \qquad  \leq C_5 (C_{8} + 1 ) 
 || a_{t}^{k} - a_{t}^{k-1} ||_{L^8}  \, 
 ( ||a_{t}^{k} ||_{L^8} + || a_{t}^{k-1} ||_{L^8}) . \\   
 \end{split}
\label{eq:ind0}
\end{equation}

We now prove the following: 
\begin{lemma}
There exists a constant  $C_9 >0$ 
independent of $t$ such that the following hold for all $k$ and $t$ 
 sufficiently small. 
\begin{equation}
|| a_{t}^{k} ||_{L^8} \leq  \, C_9 t^{\frac{1}{3}} , 
\label{eq:ind1} 
\end{equation} 
\begin{equation}
|| a_{t}^{k} -a_{t}^{k-1} ||_{L^8} \leq \, C_9 t^{\frac{1}{3}} 2^{-k} . 
\label{eq:ind2}
\end{equation}
\label{lem:ind}
\end{lemma}

\begin{proof}
The proof goes by induction. 
For $k =1$, 
\begin{equation*}
 ||a_{t}^1 ||_{L^8} \leq C_5 
||\tilde{\pi}_{7}^2 (F_{A_t}) ||_{L^4} \leq C t^{\frac{1}{3}} .
\end{equation*}
Suppose that \eqref{eq:ind2} holds for $1 , 2, \dots , k$. Then we
 obtain 
\begin{equation*}
 \begin{split}
  || a_{t}^{k} ||_{L^8} 
 &\leq || a_{t}^{1} ||_{L^8} + ||a_{t}^{1} - a_{t}^{2} ||_{L^8}  +
  \cdots + || a_{t}^{k} - a_{t}^{k-1} ||_{L^8} \\
 &\leq C_9
   t^{\frac{1}{3}} \left( \frac{1}{2} + \frac{1}{4} + \cdots + \frac{1}{2^{k-1}} \right) \\
 &\leq \, C_9 t^{\frac{1}{3}} . \\  
 \end{split}
\end{equation*}
Hence if we assume that  \eqref{eq:ind2} holds for $1 , 2, \dots , k$,
 then \eqref{eq:ind1} holds for $k$.

Now we suppose that \eqref{eq:ind1} and \eqref{eq:ind2} hold for  $1 , 2, \dots , k$. 
Then, by \eqref{eq:ind0}
\begin{equation*}
 \begin{split}
 || a_{t}^{k+1} - a_{t}^{k} ||_{L^8} 
 &\leq C 
  || a_{t}^{k} - a_{t}^{k-1} ||_{L^8}  \, 
   ( ||a_{t}^{k} ||_{L^8} + || a_{t}^{k-1} ||_{L^8})  \\
 &\leq C ( C_9 
 t^{\frac{1}{3}} 2^{-k}) 
( C_9 t^{\frac{1}{3}} + C_9 t^{\frac{1}{3}}) . 
 \end{split}
\end{equation*}
Therefore, if we take $t$ small enough so that $2 C C_9
 t^{\frac{1}{3}} \leq \frac{1}{2}$, then 
\begin{equation*}
 || a_{t}^{k+1} -a_{t}^{k} ||_{L^8} 
\leq  \, C_9 t^{\frac{1}{3}} 2^{-k-1} .
\end{equation*}
\end{proof}

Lemma \ref{lem:ind} and \eqref{eq:ind0} imply $\{ a_{t}^{k} \}$ and $\{ 
\nabla a_{t}^{k} \}$ are Cauchy sequences in $L^8$ and $L^4$ 
respectively, thus 
$\{ a_{t}^{k} \}$ 
and $\{ 
\nabla a_{t}^{k} \}$ converge to 
$a_t$ and $\nabla a_t$ in $L^8$ and $L^4$ respectively for some unique 
$a_t \in L^{4}_{1} (\mathfrak{u} (E) \otimes \Lambda^2)$. 
In addition, 
Lemmas \ref{estcoker}, \ref{lem:ind} and \eqref{eq:ind0} imply 
that there exists a constant $C_{10} >0$ such that the
following holds for all $k$ and $t$ sufficiently small: 
\begin{equation*}
 || c_{t}^{k} ||_{L^{4}_{1}} \leq C_{10} t^{\frac{1}{3}}, 
\quad || c_{t}^{k} - c_{t}^{k-1} ||_{L^{4}_{1}} \leq 
C_{10} t^{\frac{1}{3}} 2^{-k}. 
\end{equation*} 
Thus, $\{ c_{t}^{k} \}$ converges in $L^{4}_{1} ( \mathfrak{u}(E)
\otimes \Lambda^{0}) $, and hence in $C_{Z}$. 

Therefore, we obtain 
\begin{proposition} 
For $t$ sufficiently small 
there exists $a_t \in L^{4}_1 ( \mathfrak{u} (E) \otimes \Lambda^1 )$
 with $|| a_t ||_{L^8} \leq C_{9} t^{\frac{1}{3}}$ and 
$ c_{t} \in C_Z$ with 
$|| c_t ||_{L^{4}_{1}} \leq C_{10} t^{\frac{1}{3}}$  
such that $A_{t} + a_{t}$ satisfies 
the $Spin(7)$-instanton 
equation and $d_{A_t}^{*} a_t  = c_t$.    
\end{proposition}

\subsection{Regularity}

We use the elliptic theory for $L^p$ spaces \cite{GT}. 
If $D$ is an elliptic operator of order $\ell$ , then for each $k \geq
0$ 
\begin{equation*}
 ||s ||_{L^{p}_{k + \ell}} \leq 
C \left( || Ds ||_{L_{k}^{p}} 
+ ||s ||_{L^{p}} \right). 
\end{equation*}

\begin{lemma}
If $A+a 
\in L^4_1 ( \mathfrak{u} (E) \otimes \Lambda^1 )$ 
is a $Spin(7)$-instanton with  $|| a ||_{L^{8}}$ sufficiently small and
 $|| d_{A}^{*} a ||_{L^{4}_{1}}$ bounded, 
then $a \in L^4_2 ( \mathfrak{u} (E) \otimes \Lambda^1 )$. 
\end{lemma}

\begin{proof}
This follows from the standard argument (see  for example \cite{DK} 
 pp. 61--62). 
From the elliptic regularity, 
\begin{equation*}
 \begin{split}
  ||a ||_{L^4_2} 
&\leq C \left( || L_{A} a ||_{L^4_1} + ||a ||_{L^4} \right) \\ 
&\leq C 
\left( || \left( 
d_{A}^{*} a , \pi^2_7 (a \wedge a) \right)||_{L^4_1} + ||a ||_{L^4} \right) \\ 
&\leq C 
\left( || d_{A}^{*} a ||_{L^{4}_{1}} + || a ||_{L^8} \, || a ||_{L^8_1} + ||a ||_{L^4} \right) \\ 
&\leq C 
\left(  || d_{A}^{*} a ||_{L^{4}_{1}} + 
|| a ||_{L^8} \, || a ||_{L^4_2} + ||a ||_{L^4} \right) . \\ 
 \end{split}
\end{equation*}
Hence, 
\begin{equation*}
 \left( 1 - C  || a ||_{L^8} \right) 
|| a ||_{L_{2}^{4}} 
\leq C \left( || d_{A}^{*} a ||_{L^{4}_{1}} +  || a ||_{L^4} \right) . 
\end{equation*}
Therefore, if $||a ||_{L^8}$ is small enough, and $|| d_{A}^{*} a ||_{L^{4}_{1}}$ 
is bounded, then $|| a ||_{L^4_2}$ is 
bounded. 
\end{proof}

A similar argument yields $a \in L_{3}^{4}$. Using the Sobolev embedding
theorem, 
we 
obtain $a \in L_{2}^{8}$. Then one can use the argument in Section 8 of 
\cite{Lewis} to obtain the smoothness of $a$.

\begin{remark}
From the Sobolev embedding theorem, if $a \in L^4_2$, then $ a \in  L^{8}_{1}$. 
Since $Spin(7)$-instantons are Yang--Mills connections, thus we can use 
results on Yang--Mills connections such as in \cite{Uh1}, 
\cite{Uh2}, and \cite{MR2030823}. 
For example, use Theorem 9.4 in \cite{MR2030823} to find a gauge 
transformation $g$ such that $g^* (a)$ is smooth.  
\end{remark}

\section{Example}

We consider an example from \cite{MR1787733} (Example 15.7.3).  
Let $Y$ be 
a complete intersection in the weighted 
projective space $\C \mathbb{P}^{6}_{3,3,3,3, 4,4,4}$ defined by 
\begin{gather*} 
 z_{0}^{4} +  z_{1}^{4} + z_{2}^{4} + z_{3}^{4} + P (z_4 , z_5 ,z_6) = 0
 , \\
 i z_{0}^{4} - i z_{1}^{4} + 2 i z_{2}^{4} - 2 i  z_{3}^{4} + Q (z_4 , 
 z_5 ,z_6) = 0, 
\end{gather*} 
where $P (z_4 , z_5 ,z_6) , Q (z_4 , z_5 ,z_6)$ are generic homogeneous
cubic polynomials with real coefficients. 
This is a Calabi--Yau four-orbifold, 
and the singular set consists of the 9 points defined by 
\begin{equation*}
 \{ [ 0 , 0  , 0 , z_4 , z_5 , z_6  ] 
\in \C \mathbb{P}^{6}_{3, 3, 3, 3, 4, 4, 4} \, : \, 
  P (z_4 , z_4 ,z_6) = Q (z_4 , z_5 ,z_6 ) =0 \} ,
\end{equation*}
and the curve $\Sigma$ 
defined by 
\begin{equation*}
\begin{split}
 \Sigma &= \{ [z_0 , z_1 , z_2 , z_3 , 0 , 0, 0 ] \in 
\C \mathbb{P}^{6}_{3,3,3,3,4,4,4} \, 
: \,  z_{0}^{4} + z_{1}^{4} +z_{2}^{4} + z_{3}^{4} = 0,   \\
  & \qquad \qquad \qquad \qquad  
   \qquad \qquad \qquad \qquad   
  i z_{0}^{4} -i  z_{1}^{4} + 2i z_{2}^{4} - 2i  z_{3}^{4} = 0 \} . \\
\end{split}
\end{equation*}
We consider an anti-holomorphic involution $\sigma : Y \to Y$ defined by 
\begin{equation*}
 \sigma : 
[ z_0 , z_1 ,  \cdots , z_{6}] \mapsto 
[\bar{z}_1 , - \bar{z}_{0} , \bar{z}_{3} , - \bar{z}_{2} , \bar{z}_{4} , 
 \bar{z}_{5} , \bar{z}_6 ] .
\end{equation*}   
This fixes some points of the singular sets, which depend on the choice of $P$ and
$Q$. 
We take these $P$ and $Q$ so that there are five fixed points, 
say $p_1, p_2 , p_3, p_4 , p_5$,  of
$\sigma$ in the singular set, 
and two pairs of the singular points, say 
$p_6$ and  $p_7$ , $p_8$ and $p_9$, 
which are swapped with each other.

We define $Y'$ to be the blow-up of $Y$ along $\Sigma$, and lift
$\sigma$ 
to $Y'$ to give an anti-holomorphic involution $\sigma : Y' \to Y'$ 
with fixed points $p_1 , \dots , p_5$.  
The singular points of $Y' / \sigma $ are $p_1 , \dots , p_5 , 
p_6= p_7 , p_8 =p_9$. We put $X_{n_j} \, ( j = 1 \text{ or } 2)$ at 
$p_1, \dots , p_5$, and $X$ at $p_6= p_7 , p_8 =p_9$, where $X$ is 
the blow-up of $\C^{4} / \Z_4$ at the origin, 
and then apply the construction described in Section 2 
(some modifications about gluing $X$ at $p_6= p_7 , p_8 =p_9$ 
are needed, but they are trivial) 
to get a compact 
$Spin(7)$-manifold $M$.

\paragraph{Ingredient bundles. }

We consider a line bundle $L_{D}$ over $X$, which is 
determined by the exceptional
divisor $D=\C \mathbb{P}^3$. 
We equip $L_{D}$ with a Hermitian metric. 
Note that $L_{D}$ has trivial holonomy at infinity.

We put $E_{X, k} = L_{D}^{k} \oplus L_{D}^{-k} \, (k \in \Z)$  
at $p_6$ and $p_7$, 
$E_{X, \ell} = L_{D}^{\ell} \oplus L_{D}^{-\ell} \, (\ell \in \Z)$ 
at $p_8$ and $p_9$, and the rank two trivial bundle $\underline{\C}^{2}$
at each $p_1 \, \dots , p_5$,  
then glue them together 
to $\underline{\C}^{2}$ over $Y'$ 
to get a vector bundle $E$ over $M$. 
Since 
\begin{equation*}
\begin{split}
\text{Aut} &(L_{D}^{m} \oplus L_{D}^{-m})  \\
 &= H^{0} ( X , \mathcal{O}_X) 
 \oplus H^{0} (X , L_{D}^{2m}) 
\oplus H^{0} (X , L_{D}^{-2m}) 
\oplus H^{0} ( X , \mathcal{O}_X),  \\
\end{split}   
\end{equation*}
the holomorphic automorphism group of $E_{X, m}$ consists of upper triangular
matrices in 
$SU(2)$.  
Therefore, the automorphism group of $E$ is the intersection of 
contributions from the stabilizer groups of $E_{X, k}$ and
$E_{X, \ell}$ 
at each $p_{i} \, (i
= 6, 7, 8 ,9)$, schematically it is   
\begin{equation*}
\begin{split}
\left(  A_{6}^{-1} 
\left( 
\begin{matrix}
* & * \\
0 & * \\
\end{matrix}
\right) A_6
\right) 
&\cap 
\left(  A_{7}^{-1} 
\left( 
\begin{matrix}
* & * \\
0 & * \\
\end{matrix}
\right) A_7
\right) \\
& \quad \cap 
\left(  A_{8}^{-1} 
\left( 
\begin{matrix}
* & * \\
0 & * \\
\end{matrix}
\right) A_8 
\right)
\cap 
\left(  A_{9}^{-1} 
\left( 
\begin{matrix}
* & * \\
0 & * \\
\end{matrix}
\right) A_9
\right), \\
\end{split}
\end{equation*}
where $A_6 , A_7 , A_8 ,A_9 \in SU(2)$, and 
$A_6$ and $A_{7}$, $A_8$ and $A_{9}$ are conjugate.  
This becomes $c \cdot \text{id} \, (c \in \C^{*})$ 
for generic $A_6$ and $A_8$. 
Thus, we can make the
resulting vector bundle $E$ irreducible.

\paragraph{Hermitian--Einstein connections. }

We equip $L_{D}^{m}$ with a Hermitian--Einstein connection. 
Since $c_1 (L_{D}^{m})$ lies in the image of $H^{2}_{cs} (X) \to H^2
(X)$, 
thus, from Lockhart (\cite{MR879560},  Section 8), there exists a unique 
harmonic 2-form $\alpha$ in $X$ such that $[\alpha] = 2 \pi c_1
(L_{D}^{k})$ 
and $\alpha = O( r^{-7} )$. 
We then decompose $\alpha$ into 
\begin{equation*}
 \alpha = \alpha^{2,0} + \alpha^{0,2} + \alpha_{0}^{1,1} + (\alpha \cdot
  \omega) \omega . 
\end{equation*} 
Since $H^{0,2} (X) =0$, $\alpha^{0,2} = \alpha^{2,0} =0 $. 
Moreover, since $(\alpha \cdot \omega)$ is harmonic, vanishing at
infinity, 
thus, $(\alpha \cdot  \omega) =0 $ by the maximum principle.  
Hence $\alpha = \alpha_{0}^{1,1}$. 
We now take a connection $A$ of $L_{D}^{m}$ with $F_{A} = \alpha$, then 
this $A$ is a Hermitian--Einstein connection on $L_{D}^{m}$.

\paragraph{The conditions for the linearized operators. }

We examine the conditions for the linearized operators in Section 4.1. 
For the $Spin(7)$-orbifold side, we have $H^{2}_{7} (Z) =0$, since 
$Z$ has holonomy 
$SU(4) \rtimes \Z_{2}$. 
Thus, the cohomology $H^{2} (Z , \mathfrak{su} (E))$ of the complex 
\eqref{eq:complex} vanishes, as 
 $H^{2} (Z , \mathfrak{su} (\underline{\C}^2)) = 
H^{2}_{7} (Z) \otimes \mathfrak{su} (\underline{\C}^2)$. 
Hence, $\ker L_{A_Z}^{*}$ lies in $\Omega^{0} ( Z , \mathfrak{su} (
\underline{\C}^2))$. 

For the ALE side, 
we introduce a sheaf cohomology on $X$ with the decay rate $\delta$ at 
infinity as follows. 
Let $(L_{D}^{m})_{\delta}$ be 
a sheaf of holomorphic sections of $L_{D}^{m}$ 
with the decay rate $\delta$. 
We consider the following injective resolution 
of $ (L_{D}^{m})_{\delta}$: 
\begin{equation}
\begin{split}
0 &\longrightarrow 
 (L_{D}^{m})_{\delta} \xrightarrow{ \, \, \, \, \, i   \, \, \, \, \,} 
\Omega^{0} (L_{D}^{m})_{\delta} 
\xrightarrow{ \, \, \, \, \, \bar{\partial} \, \, \, \, \,} 
\Omega^{0, 1} (L_{D}^{m})_{\delta - 1} 
\xrightarrow{ \, \, \, \, \, \bar{\partial}  \, \, \, \, \,} 
\Omega^{0,2} (L_{D}^{m})_{\delta - 2} \\ 
& \qquad \qquad \qquad 
 \xrightarrow{ \, \, \, \, \, \bar{\partial}  \, \, \, \, \,} 
\Omega^{0,3 } (L_{D}^{m})_{\delta -3}
\xrightarrow{ \, \, \, \, \, \bar{\partial}  \, \, \, \, \,} 
\Omega^{0, 4} (L_{D}^{m})_{\delta -4}   
\longrightarrow 0 .  
\end{split}
\label{eq:injres}
\end{equation}
We then have a complex induced by \eqref{eq:injres}: 
\begin{equation}
\begin{split}
0 &\longrightarrow 
C^{\infty} ( X , \Omega^{0} (L_{D}^{m})_{\delta}) 
\xrightarrow{ \, \, \, \, \, \bar{\partial}_{X} \, \, \, \, \,} 
C^{\infty} ( X , \Omega^{0, 1} (L_{D}^{m})_{\delta - 1}) 
\xrightarrow{ \, \, \, \, \, \bar{\partial}_{X}  \, \, \, \, \,} \\
&  \qquad  \cdots  
 \xrightarrow{ \, \, \, \, \, \bar{\partial}_{X}  \, \, \, \, \,} 
C^{\infty} ( X  , \Omega^{0,3 } (L_{D}^{m})_{\delta -3}) 
  \xrightarrow{ \, \, \, \, \, \bar{\partial}_{X}  \, \, \, \, \,} 
C^{\infty} ( X  , \Omega^{0, 4} (L_{D}^{m})_{\delta -4} ) 
\longrightarrow 0 .  
\end{split}
\label{eq:compde}
\end{equation} 
We denote by $H_{\delta}^{i} (X , L_{D}^{m})$ the 
$i$-th cohomology of the complex \eqref{eq:compde} for each 
$i= 0, \dots , 4 $.

\begin{lemma}
\begin{equation*}
 H^{0}_{\delta} (X , L_{D}^{m}) = 0 
\end{equation*}
for all $m \in \Z$ and $\delta <0$. 
\end{lemma}

\begin{proof}
We take the standard holomorphic section $s_D$ of $L_{D}$, 
and write $t \in H^{0}_{\delta} (L^{m})$ as $t= f s_D$, where $f$ is a 
meromorphic function on $X$. 
Since $t$ is holomorphic, $f$ has a pole of the order $\leq m$ if 
$m \geq0 $, or $f$ has a zero of the order $\geq - m$ if $m \leq 0$ at 
$D$. 
By Hartogs' theorem, a holomorphic function on $\C^4 \setminus 0$ 
extends to $\C^4$, hence $f$ has no pole at $D$, that is, $f$ is 
holomorphic.  
Since we impose growth condition at infinity, $|f| \to 0$ as $r \to \infty$. 
Hence $f \equiv 0$ by the maximum principle. 
\end{proof}

\begin{lemma}
\begin{equation*}
H^{2}_{\delta} (X , L_{D}^{m}) = 0 
\end{equation*}
for all $m \in \Z$ and $\delta < 0$.
\end{lemma}

\begin{proof}
We consider the following exact sequence,  
\begin{equation*}
 0 \longrightarrow (L_{D}^{-1})_{\delta} \longrightarrow 
 (\mathcal{O}_{X})_{\delta} \longrightarrow 
\mathcal{O}_{D} \longrightarrow 0. 
\end{equation*}
Twisting by $L_{D}^{m} \, (m \in \Z_{\geq 0})$, we obtain 
\begin{equation}
 0 \longrightarrow ( L_{D}^{m-1})_{\delta} \longrightarrow
 (L_{D}^{m})_{\delta} 
\longrightarrow \mathcal{O}_{D} \otimes L_{D}^{m} \longrightarrow 0. 
\label{eq:shortex}
\end{equation}
From this, we get a long exact sequence: 
\begin{equation*}
\begin{split}
\cdots 
&\longrightarrow H^{1} (D , \mathcal{O}_{D} (-4m) ) 
\longrightarrow H^2_{\delta} (X, L_{D}^{m-1}) \\
&\quad \quad 
\longrightarrow H^2_{\delta} (X, L_{D}^{m}) 
\longrightarrow H^2 (D, \mathcal{O}_{D} (-4m ))  
\longrightarrow \cdots , \\
\end{split}
\end{equation*}
where we used $H^{i} (X , \mathcal{O}_{D} \otimes L_{D}^{m}) \cong 
H^{i} ( D , \mathcal{O}_{D} (-4m)) $. 
As  $H^{1} (D , \mathcal{O}_{D} ( -4m )) = 
H^{2} (D , \mathcal{O}_{D} ( - 4m )) =0$, 
we get an isomorphism $ H^2_{\delta} (X, L_{D}^{m-1}) \to  
H^2_{\delta} (X, L_{D}^{m})$. 
In addition, we have $H^2_{\delta} (X,
\mathcal{O}_{X}) =0 $ for $\delta <0$ by Theorem 5.3 in \cite{MR1824171}. 
Hence $H^2_{\delta} (X, L_{D}^{m})=0$ for $m \in
\Z_{\geq 0}$ by induction. 
The dual argument yields $H^2_{\delta} (X, L_{D}^{m})=0$ for 
$m \in \Z_{\leq 0}$ as well. 
\end{proof}

Hence, 
\begin{equation*}
\begin{split}
 H^{2}_{\delta} &(X , \text{End} (E_{X, m}))  \\
 &= H^{2}_{\delta} ( X, \mathcal{O}_X)  \oplus 
   H^{2}_{\delta} ( X, L_{D}^{2m}) 
\oplus H^{2}_{\delta} (X , L_{D}^{-2m}) 
\oplus H^{2}_{\delta} ( X, \mathcal{O}_X)=0  \\
\end{split} 
\end{equation*}
for $\delta < 0$. 
Thus, the linearized operator 
$L_{A_{X}} : L_{1, \delta}^{4}
(\mathfrak{u} ( E_{X, m}) \otimes \Lambda^{1} (X) ) 
\to L_{\delta -1}^{4}  
( \mathfrak{u} (E_{X, m}) \otimes ( \Lambda^{0} (X) \oplus \Lambda_{7}^{2} 
(X) ) )
$ 
with $\delta \in ( - 7, 0)$ is surjective, thus the condition in
Section 4.1 is satisfied. 
Therefore, we obtain a $Spin(7)$-instanton on this $E$ by Theorem 6.1.

Furthermore, in this example, 
$\dim C_Z =3$ and $\dim K_Z =0$, 
and we also have the following for the ALE side: 
\begin{lemma}
\begin{equation}
H^{1}_{\delta} (X , L_{D}^{m}) = 
\begin{cases} 
\C^{\frac{1}{3} m^2 ( 8m^2 -5)}  & (m \leq  0) , \\
0 & (m > 0 ). \\
\end{cases}
\label{eq:h1}
\end{equation}
\begin{equation}
H^{3}_{\delta} (X , L_{D}^{m}) = 
\begin{cases} 
\C^{\frac{1}{3} m^2 ( 8m^2 -5)}  & (m > 0) , \\
0 & (m \leq 0 ). \\
\end{cases}
\label{eq:h3}
\end{equation}
\end{lemma}

\begin{proof}
We again use the long exact sequence induced by \eqref{eq:shortex}: 
\begin{equation*}
\begin{split}
\cdots &\longrightarrow H^0_{\delta} (X, L_{D}^{m}) 
\longrightarrow H^{0} (D , \mathcal{O}_{D} (-4m) ) 
\longrightarrow H^{1}_{\delta} (X, L_{D}^{m-1}) \\
&\quad \quad \quad \quad \quad \quad \quad \quad 
\longrightarrow H^1_{\delta} (X, L_{D}^{m}) 
\longrightarrow H^1 (D, \mathcal{O}_{D} (-4m ))  
\longrightarrow \cdots .  \\
\end{split}
\end{equation*}
Since $H^{0}_{\delta} (X , L_{D}^{m}) = H^{1} (D , \mathcal{O}_{D} ( -4 m)) = 0$, 
and 
\begin{equation*}
 \dim  H^{0} 
(D , \mathcal{O}_{D} (-4m) ) = 
\begin{cases}
\frac{(3-4m) (2 -4m) (1 -4m)}{6} & ( m \leq 0) , \\
0 & ( m > 0 ) ,\\ 
\end{cases}
\end{equation*}
we get 
\begin{equation*}
 \dim  H^{1}_{\delta}  
(X , L_{D}^{m} ) = 
\dim H^{1}_{\delta} 
(X, L_{D}^{m-1}) - 
\begin{cases}
\frac{(3-4m) (2 -4m) (1 -4m)}{6} & ( m \leq 0) , \\
0 & ( m > 0 ) . 
\label{eq:rec} 
\end{cases}
\end{equation*}
From this with $\dim H^{1}_{\delta} (X , \mathcal{O}_{X}) =0$, 
we obtain \eqref{eq:h1} by induction, and  
\eqref{eq:h3} follows from \eqref{eq:h1} either by Serre duality,   
or by the same method of proof using $\dim H^{3} (D , \mathcal{O}_{D} (-4m)) 
= - \frac{(3-4m)(2- 4m)(1-4m)}{6}$ if $m > 0$ and $0$ if $m \leq 0$.  
\end{proof}

Therefore, the real dimension of $K_{X_{n_j}}$ is $\frac{4 k^2}{3}  (32
k^2 -5)$ at $p_6 =p_7$, and 
$\frac{4 \ell^2}{3} (32
\ell^2 -5)$ at $p_8 =p_9$. 
We also have $\dim C_{Z} =3$ and the dimensions of all the other spaces in $K$ are zero. 
Hence, \eqref{eq:indexc} shows that the virtual dimension of the moduli
space in this example is given by 
$$ -3 + \frac{4 k^2 }{3}  (32 k^2 -5)  + \frac{ 4 \ell^2}{3} (32 \ell^2
-5) .   
$$
When $k=\ell=0$, we obtain the trivial, flat $SU(2)$ instanton, which is
rigid with automorphism group $SU(2)$ of dimension 3, 
and the virtual dimension is $-3$. 
For $k, \ell$ not both zero, we get a positive dimensional
moduli space,  
and for the generic gluing data,  the solution given by Theorem 6.1 
is unobstructed and irreducible, and the moduli space
is smooth of the given dimensions near the solution.


\begin{flushleft}
E-mail: yu2tanaka@gmail.com
\end{flushleft}


\end{document}